\documentclass[pdflatex,sn-mathphys-num]{sn-jnl}% Math and Physical Sciences Numbered Reference Style

\usepackage{graphicx}%
\usepackage{multirow}%
\usepackage{amsmath,amssymb,amsfonts}%
\usepackage{amsthm}%
\usepackage{mathrsfs}%
\usepackage[title]{appendix}%
\usepackage{xcolor}%
\usepackage{textcomp}%
\usepackage{manyfoot}%
\usepackage{booktabs}%
\usepackage{algorithm}%
\usepackage{algorithmicx}%
\usepackage{algpseudocode}%
\usepackage{listings}%
\usepackage{placeins}%
\usepackage{pifont}%
\theoremstyle{plain}
\newtheorem{theorem}{Theorem}[section]
\newtheorem*{theorem*}{Theorem}
\newtheorem{proposition}[theorem]{Proposition}
\newtheorem{lemma}[theorem]{Lemma}
\newtheorem{corollary}[theorem]{Corollary}

\theoremstyle{remark}

\theoremstyle{definition}
\newtheorem{definition}[theorem]{Definition}
\newtheorem{notation}[theorem]{Notation}

\numberwithin{equation}{section}

\renewcommand{\orcidlogo}{%
  \includegraphics[width=10pt]{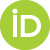}%
}

\begin{document}

\title[The K-theory of uniform Roe algebras for coarse structures generated by finite-rank free abelian subgroups]{The K-theory of uniform Roe algebras for coarse structures generated by finite-rank free abelian subgroups}

\author*{\fnm{Charles} \sur{Fanning} \orcid{https://orcid.org/0009-0002-8251-2802}}\email{cfannin8@students.kennesaw.edu}

\author{\fnm{Mehmet} \sur{Aktas} \orcid{https://orcid.org/0000-0002-9527-9600}}\email{maktas1@kennesaw.edu}

\affil{\orgdiv{School of Data Science and Analytics}, 
\orgname{Kennesaw State University}, 
\orgaddress{\street{1000 Chastain Rd NW}, 
\city{Kennesaw}, 
\postcode{30144}, 
\state{Georgia}, 
\country{United States}}}

% \author*[1,2]{\fnm{First} \sur{Author}}\email{iauthor@gmail.com}



% \affil*[1]{\orgdiv{Department}, \orgname{Organization}, \orgaddress{\street{Street}, \city{City}, \postcode{100190}, \state{State}, \country{Country}}}

% \affil[2]{\orgdiv{Department}, \orgname{Organization}, \orgaddress{\street{Street}, \city{City}, \postcode{10587}, \state{State}, \country{Country}}}

% \affil[3]{\orgdiv{Department}, \orgname{Organization}, \orgaddress{\street{Street}, \city{City}, \postcode{610101}, \state{State}, \country{Country}}}

\abstract{

For a uniformly locally finite coarse space $X$, the uniform Roe algebra $C_u^*(X)$ is the operator norm closure of the controlled operators on $\ell^2(X)$. The $K$-theory of uniform Roe algebras is known in asymptotic dimension zero, but it is not fully understood in higher dimensions. We compute $K_0(C_u^*(G,\mathcal E))$ and $K_1(C_u^*(G,\mathcal E))$ for every countable discrete abelian group $G$ and every finite-rank free abelian subgroup $H\leq G$, where $\mathcal E$ is the coarse structure generated by $H$. We use the Proietti--Yamashita spectral sequence to express the $K$-theory in terms of $H_*(H;\ell^\infty(G,\mathbb Z))$, which we then compute.

}

% Give between 4 and 6 keywords.
% \keywords{keyword1, Keyword2, Keyword3, Keyword4}
% \pacs[MSC Classification]{43A25, 43A35, 43A45, 43A05, 43A07}

\maketitle

% Optional:
% \tableofcontents

\section*{Introduction}\label{sec1}

For a uniformly locally finite coarse space \( X \), the uniform Roe algebra \( C_u^*(X) \) is the operator norm closure of the controlled operators on \( \ell^2(X) \). For a countable discrete abelian group \( G \) and a finite-rank free abelian subgroup \( H\leq G \), we let \( \mathcal E \) be the coarse structure generated by the sets \( \{(g,k)\in G\times G:g-k\in F\} \), where \( F\subseteq H \) is finite. The controlled operators for \( \mathcal E \) preserve each \( H \)-coset and have uniformly bounded propagation within that coset.

% === %

A central theme in coarse index theory is to understand how the large-scale geometry and topology of a space determine the \( K \)-theory of its Roe and uniform Roe algebras \cite{roe1993coarse,higson1994homotopy,yu1997localization,skandalis2002coarse,Spakula2009,engel2019uniform,engel2019rough,bunke2020coarse}. More recently, this program has increasingly used homology, especially uniformly finite homology and groupoid homology, to study the \( K \)-theory of uniform Roe and groupoid \( C^* \)-algebras \cite{matui2012homology,bonicke2023dynamic,manuilov2024mapping,krutoy2026bijective}.

The \( K \)-theory of uniform Roe algebras is understood most explicitly for spaces of low asymptotic dimension, where the literature obtains computations in asymptotic dimension zero and strong results on \( K_0 \) in asymptotic dimension at most one \cite{winter2010nuclear,li2018classification,li2018low,chung2021structure,bonicke2023dynamic}. The low-dimensional description of \( K_0 \) already fails for some higher-dimensional spaces, and recent work develops broader results using groupoid homology and comparison maps and for \( \ell^p \) uniform Roe algebras \cite{li2018low,chung2021structure,bonicke2023dynamic,manuilov2024mapping,krutoy2026bijective}.

The motivation for using groupoid homology is that it provides a method for computing the \( K \)-theory of reduced groupoid \( C^* \)-algebras \cite{matui2016etale,farsi2018ample,yi2020homology,proietti2022homology,bonicke2023dynamic,proietti2023homology}.

% === %

The objective of this paper is to compute \( K_*(C_u^*(G,\mathcal E)) \) for every countable discrete abelian group \( G \) and every finite-rank free abelian subgroup \( H\leq G \).

% \paragraph{Main Result}

% ---------------------------------------- %
% (i) Main result one                      %
% ---------------------------------------- %
% ---------------------------------------- %
%     (a) Statement of the main result     %
% ---------------------------------------- %
% First, a paragraph, then the theorem in a 
% theorem block. The first paragraph is 
% structured as so:
%
% > First instance:
% "We first [verb (e.g., show, identify, 
% compute, etc.)] in Theorem~\ref{thm:...} 
% that ... "
%
% > Between first and last:
% "We next ...
%
% > Last instance: 
% "We then [verb] in Theorem~\ref{thm...}
% that ...
%
% > If there is only one main theorem:
% "We [verb] in Theorem~\ref{thm:...} that...
% "We [verb] (object) in Theorem~\ref{thm...}...
%
% This should only be one sentence stating
% clearly and concisely but still rigorously
% what the theorem does.
%
% If there are yet-to-be-defined terms in
% the theorem statement, add a sentence
% after the one sentence in the paragraph
% preceding the theorem block saying
% "Here, [(e.g., A  = B, by C we mean
% D, and ...)]

We compute both \( K \)-groups in Theorem~\ref{thm:finite-rank-k-theory} in terms of \( n=\operatorname{rank}(H) \), the quotient \( G/H \), and \( \mathfrak c=2^{\aleph_0} \). Here, \( \mathbb Q^{(\mathfrak c)} \) denotes the direct sum of \( \mathfrak c \) copies of \( \mathbb Q \).

The following is Theorem~\ref{thm:finite-rank-k-theory}.

\begin{theorem*}
The \( K \)-groups of \( C_u^*(G,\mathcal E) \) satisfy
\[
\left(
K_0\bigl(C_u^*(G,\mathcal E)\bigr),
K_1\bigl(C_u^*(G,\mathcal E)\bigr)
\right)
\cong
\begin{cases}
\left(
\ell^\infty(G,\mathbb Z),
0
\right),
& n=0,\\[1ex]
\left(
\mathbb Q^{(\mathfrak c)},
\ell^\infty(G/H,\mathbb Z)
\right),
& n=1,\\[1ex]
\left(
\mathbb Q^{(\mathfrak c)}
\oplus
\ell^\infty(G/H,\mathbb Z),
\mathbb Q^{(\mathfrak c)}
\right),
& n\geq2 \text{ and } n \text{ is even},\\[1ex]
\left(
\mathbb Q^{(\mathfrak c)},
\mathbb Q^{(\mathfrak c)}
\oplus
\ell^\infty(G/H,\mathbb Z)
\right),
& n\geq3 \text{ and } n \text{ is odd}.
\end{cases}
\]
The isomorphism is canonical when \( n=0 \) and noncanonical when \( n\geq1 \). The associated graded group in degree \( n \) is canonically isomorphic to \( \ell^\infty(G/H,\mathbb Z)\otimes_{\mathbb Z}\Lambda^nH \). An ordered basis of \( H \) identifies this group with \( \ell^\infty(G/H,\mathbb Z) \).
\end{theorem*}

% ---------------------------------------- %
%     (b) Main idea behind main result     %
% ---------------------------------------- %
% A one sentence paragraph which explains
% what each individual component in the 
% theorem block actually does.
% Typically begins with "That is, ..." 
% unless we do not fully characterize
% every component in the theorem, in which
% case we don't use a transition.

That is, the nonzero associated graded groups of the filtration of \( K_*(C_u^*(G,\mathcal E)) \) in degrees below \( n \), grouped according to parity, have direct sums isomorphic to copies of \( \mathbb Q^{(\mathfrak c)} \), and the associated graded group in degree \( n \) occurs in \( K_0 \) when \( n \) is even and in \( K_1 \) when \( n \) is odd.

% ---------------------------------------- %
%     (c) Outline of proof of main result  %
% ---------------------------------------- %
% The role of the proof outline is to
% introduce just enough of the machinery
% to show the reader the main insight which
% enables the proof.
%
% One short paragraph explaining why the
% computation is possible. State the main
% reduction and only the information needed
% to show how it yields the theorem. Display
% the central equation when it makes this
% reduction substantially clearer.

The computation of \( K_*(C_u^*(G,\mathcal E)) \) reduces to group homology through the Proietti--Yamashita spectral sequence \cite[Theorem~3.3]{proietti2022homology}, whose \( E^2 \)-page is
\[
E^2_{p,q}
\cong
\begin{cases}
H_p\bigl(H;\ell^\infty(G,\mathbb Z)\bigr),
& q \text{ is even},\\
0,
& q \text{ is odd}.
\end{cases}
\]
We show that \( H_n(H;\ell^\infty(G,\mathbb Z))\cong\ell^\infty(G/H,\mathbb Z)\otimes_{\mathbb Z}\Lambda^nH \) and that \( H_p(H;\ell^\infty(G,\mathbb Z))\cong\mathbb Q^{(\mathfrak c)} \) for \( p<n \). The higher differentials in the spectral sequence vanish, so \( E^\infty=E^2 \).

% \paragraph{Importance}

% ---------------------------------------- %
% (i) Is it surprising?                  %
% ---------------------------------------- %
% Should begin with: "What is perhaps
% surprising is that ... "
%
% This paragraph should be one short sentence
% with no \cite.

What is perhaps surprising is that all higher differentials in the Proietti--Yamashita spectral sequence vanish.

% ---------------------------------------- %
% (i) What do we assume?                   %
% ---------------------------------------- %
% The first sentence (the introduction) should
% be: "We make the following assumptions throughout
% this paper."
%
% Sentence 1 (excluding introductory sentence):
% "We (often) assume that [X]"
% "We (often) assume that [X] and [Y]"
% "We (often) assume that [X], [Y], and [Z]"
%
% Sentences 2+: "We also (often) assume that ..."
% "We also (often) assume that [X] and [Y]"
% "We also (often) assume that [X], [Y], and [Z]"
%
% Prefer "[X] and [Y]" or "[X], [Y], and [Z]"
% for related assumptions which can be stated more
% concisely together than separate. Prefer separate
% sentences for unrelated assumptions or assumptions
% which are too long to comfortably include in a
% sentence.
%
% This should be exactly one paragraph.

We make the following assumptions throughout this paper. We assume that \( G \) is a countable discrete abelian group and that \( H\leq G \) is a finite-rank free abelian subgroup. We also assume that \( \mathcal E \) is the coarse structure generated by such a subgroup \( H \).

% ---------------------------------------- %
% (ii) How could this be extended?         %
% ---------------------------------------- %
% "A natural next question is ... "
% This paragraph should not exceed one
% sentence in length. If there are several
% clear extensions, we should prioritize
% the clearest one in the introduction and
% save the rest for the conclusion.
%
% This should be a one clause sentence. No need to
% explain why the natural next question is natural.

A direction for future work is to compute \( K_*(C_u^*(G,\mathcal E)) \) when \( H\leq G \) is an infinite-rank free abelian subgroup.

% ---------------------------------------- %
% (iii) Is it new? and organization.       %
% ---------------------------------------- %
% 
% The first sentence should be:
% "The paper is organized as follows."
%
% Better, if we can point to a specific section,
% then this can double as the organization of the
% paper. For example:
% 1. "We introduce in Section~\ref{sec:...} ...";
% 2. "We next introduce in section~\ref{sec:...} ...";
% 3. "We then introduce in section~\ref{:...} ...".
% That is, this can be an (at most) three sentence
% paragraph (after the introductory sentence). If only 
% two sentences, use structures 1 and 3. If only one 
% sentence, use structure one.
%
% These variations are also acceptable:
% - "In Section~\ref{sec:...} we introduce ..." only for
%   structure 1.
% - "introduce" may be replaced with "show" when we wish
%   to highlight a result which is only interesting as an
%   intermediate proof step, not as a standalone result.
%
% This should be one paragraph which higlights what 
% it is that we have done which no other published work 
% has ever done before.

The paper is organized as follows. In Section~\ref{sec:group-homology-of-ell-infinity(G,Z)}, we introduce an explicit formula for \( H_p(H;\ell^\infty(G,\mathbb Z)) \) for every subgroup \( H \cong \mathbb Z^n \leq G \) and every \( n, p\geq0 \). We next introduce, in Section~\ref{sec:the-spectral-sequence-after-taking-directed-colimits}, a spectral sequence whose \( E^2 \)-page is given by the group homology \( H_p(H;\ell^\infty(G,\mathbb Z)) \) and which converges to \( K_*(C_u^*(G,\mathcal E)) \), obtained from the Proietti--Yamashita spectral sequences by taking their directed colimit. We then introduce, in Section~\ref{sec:finite-rank-computation}, an explicit formula for \( K_*(C_u^*(G,\mathcal E)) \).

\section{Preliminaries}
\label{sec:preliminaries}

The remainder of the paper uses methods from coarse geometry and operator algebras. Since these subjects use various notational conventions in the literature, we first establish the notation used throughout the paper. We then recall the results needed in the proof of the main theorem.

\begin{notation}
\label{def:standing-notation}

We use the following notation and conventions throughout the paper.

\begin{itemize}
    \item For \( C^* \)-algebras \( A \) and \( B \), we write \( A\otimes B \) for their minimal tensor product.

    \item For a Hilbert space \( V \), we write \( \mathcal B(V) \) for the \( C^* \)-algebra of the bounded operators on \( V \).

    \item For an action \( \alpha\colon H\curvearrowright A \) of a countable discrete group \( H \) on a \( C^* \)-algebra \( A \), we write \( A\rtimes_{\alpha,r}H \) for the reduced crossed product. When the action is clear from the context, we write \( A\rtimes_r H \).

    \item For a second-countable Hausdorff ample groupoid \( \mathcal G \) and a \( \mathcal G \)-\( C^* \)-algebra \( A \), we write \( \mathcal G\ltimes_r A \) for the reduced crossed product.

    \item For a groupoid \( \mathcal G \), we write \( \mathcal G^{(0)} \) for its unit space. If \( \mathcal G \) is second-countable, Hausdorff, and ample, we write \( C_r^*(\mathcal G) \) for its reduced groupoid \( C^* \)-algebra.

    \item For a groupoid \( \mathcal G \) and \( \mathcal G \)-\( C^* \)-algebras \( A \) and \( B \), we write \( KK^{\mathcal G}(A,B) \) for the equivariant Kasparov group.

    \item For a countable discrete group \( H \), a right \( \mathbb Z H \)-module \( M \), and \( p\geq 0 \), we write \( H_p(H;M) \) for the \( p \)-th group homology group of \( H \) with coefficients in \( M \).

    \item For a set \( G \), we write \( \ell^\infty(G)=\ell^\infty(G,\mathbb C) \).

    \item For a second-countable Hausdorff ample groupoid \( \mathcal G \), a \( \mathcal G \)-equivariant sheaf \( M \), and \( p\geq 0 \), we write \( H_p(\mathcal G;M) \) for the \( p \)-th groupoid homology group of \( \mathcal G \) with coefficients in \( M \). We write \( H_p(\mathcal G)=H_p(\mathcal G;\mathbb Z) \).

    \item For a \( C^* \)-algebra \( A \) and \( q\in\mathbb Z \), we write \( K_q(A) \) for the \( q \)-th periodic complex \( K \)-theory group of \( A \). If \( A \) is a \( \mathcal G \)-\( C^* \)-algebra, the same notation denotes the associated \( \mathcal G \)-equivariant \( K \)-theory sheaf when it appears as a coefficient in groupoid homology.
\end{itemize}
\end{notation}

We begin by recalling the definition of a coarse space and the associated notion of uniform local finiteness.

\begin{definition}
\label{def:coarse-space}
A \emph{coarse structure} on a set \( X \) is a collection \( \mathcal E \) of subsets of \( X\times X \) satisfying the following conditions.
\begin{enumerate}
    \item The diagonal \( \Delta_X=\{(x,x):x\in X\} \) belongs to \( \mathcal E \).
    \item If \( E\in\mathcal E \), then \( E^{-1}\in\mathcal E \), where \( E^{-1}=\{(y,x):(x,y)\in E\} \).
    \item If \( E\in\mathcal E \) and \( F\subseteq E \), then \( F\in\mathcal E \).
    \item If \( E,F\in\mathcal E \), then \( E\cup F\in\mathcal E \).
    \item If \( E,F\in\mathcal E \), then \( E\circ F\in\mathcal E \), where 
    \[ 
    E\circ F=\{(x,z):\text{there exists }y\in X\text{ such that }(x,y)\in E\text{ and }(y,z)\in F\} .
    \]
\end{enumerate}
We call the elements of \( \mathcal E \) \emph{controlled sets} and the pair \( (X,\mathcal E) \) a \emph{coarse space}. We say that \( (X,\mathcal E) \) is \emph{uniformly locally finite} if, for every \( E\in\mathcal E \), one has \( \sup_{x\in X}|E[x]|<\infty \), where \( E[x]=\{y\in X:(x,y)\in E\} \).
\end{definition}

Let \( G \) be a countable discrete abelian group and let \( H\leq G \) be free abelian of finite rank \( n \). Equip \( G \) with the coarse structure \( \mathcal E \) generated by the sets \( \{(g,k)\in G\times G:g-k\in F\} \), where \( F \) ranges over the finite subsets of \( H \). The coarse space \( (G,\mathcal E) \) is uniformly locally finite.

Let \( (\delta_g)_{g\in G} \) denote the standard orthonormal basis of \( \ell^2(G) \). Every operator \( T\in\mathcal B(\ell^2(G)) \) has matrix coefficients \( \langle T\delta_k,\delta_g\rangle \). The support of \( T \) is
\[
\operatorname{supp}(T)=\{(g,k)\in G\times G:\langle T\delta_k,\delta_g\rangle\neq0\}.
\]
If \( \operatorname{supp}(T)\in\mathcal E \), we say that \( T \) has controlled propagation. Equivalently, \( T \) has controlled propagation if and only if there exists a finite subset \( F\subseteq H \) such that \( \langle T\delta_k,\delta_g\rangle=0 \) whenever \( g-k\notin F \). The uniform Roe algebra of the coarse space \( (G,\mathcal E) \) is the operator norm closure in \( \mathcal B(\ell^2(G)) \) of the \( ^* \)-algebra of controlled propagation operators. We denote it by \( C_u^*(G,\mathcal E) \), or simply by \( C_u^*(G) \) when no confusion can arise.

The right action of \( H \) on \( G \) is given by \( g\cdot h=g+h \). This induces right actions on \( \ell^\infty(G,\mathbb C) \) and \( \ell^\infty(G,\mathbb Z) \) by \( (f\cdot h)(g)=f(g-h) \). We write \( \alpha_h(f)=f\cdot h \) for the automorphism used in the crossed product \( \ell^\infty(G)\rtimes_{\alpha,r}H \). More generally, if \( X \) is a topological right \( H \)-space, then the induced right action on \( C(X,\mathbb Z) \) is given by \( (f\cdot h)(x)=f(x\cdot(-h)) \).

The proof of the main theorem uses the following spectral sequence of Proietti and Yamashita \cite[Theorem~3.3]{proietti2022homology}.

\begin{proposition}
\label{prop:spectral-sequence}
Let \( \mathcal G \) be a second-countable Hausdorff ample groupoid with torsion-free stabilizers, and suppose that \( \mathcal G \) satisfies the strong Baum--Connes conjecture. Let \( A \) be a separable \( \mathcal G \)-\( C^* \)-algebra. Then there is a convergent spectral sequence
\begin{equation}
\label{eq:proietti-yamashita}
E^2_{p,q}
\cong
H_p(\mathcal G;K_q(A))
\Longrightarrow
K_{p+q}(\mathcal G\ltimes_r A),
\qquad
p\geq 0,\quad q\in\mathbb Z.
\end{equation}
\end{proposition}

The proof of the main theorem requires two consequences of Proposition~\ref{prop:spectral-sequence}. The first is its specialization to reduced groupoid \( C^* \)-algebras. The second is the functoriality of the spectral sequence under \'{e}tale correspondences.

Applying Proposition~\ref{prop:spectral-sequence} to \( A=C_0(\mathcal G^{(0)}) \), the isomorphism \( \mathcal G\ltimes_r C_0(\mathcal G^{(0)})\cong C_r^*(\mathcal G) \) identifies the abutment with \( K_{p+q}(C_r^*(\mathcal G)) \). By \cite[Proposition~3.1]{proietti2022homology}, the coefficient sheaf associated with \( K_q(C_0(\mathcal G^{(0)})) \) is the constant sheaf with fiber \( K_q(\mathbb C) \). Proposition~\ref{prop:spectral-sequence} therefore specializes to the convergent spectral sequence
\begin{equation}
\label{eq:groupoid-homology-spectral-sequence}
E^2_{p,q}(\mathcal G)
\cong
H_p(\mathcal G;K_q(\mathbb C))
\Longrightarrow
K_{p+q}(C_r^*(\mathcal G)).
\end{equation}
By Bott periodicity, \( K_q(\mathbb C)\cong\mathbb Z \) for even \( q \) and \( K_q(\mathbb C)=0 \) for odd \( q \). Therefore,
\begin{equation}
\label{eq:groupoid-homology-e2}
E^2_{p,q}(\mathcal G)
\cong
\begin{cases}
H_p(\mathcal G), & q \text{ is even},\\
0, & q \text{ is odd}.
\end{cases}
\end{equation}
Every differential \( d_2 \) has either a trivial domain or a trivial codomain. Hence \( d_2=0 \), and therefore \( E^2(\mathcal G)=E^3(\mathcal G) \).

The proof of the main theorem also uses the following consequence of \cite[Proposition~3.22, Remark~4.6, Theorem~5.14, Theorem~6.1]{miller2025isomorphisms}.

\begin{proposition}
\label{prop:miller-functoriality}
Let \( \mathcal G \) and \( \mathcal H \) be second-countable Hausdorff ample groupoids, and let \( \Omega \colon \mathcal G \to \mathcal H \) be a second-countable Hausdorff \'{e}tale correspondence.

Let \( A \) be a separable \( \mathcal G \)-\( C^* \)-algebra, let \( B \) be a separable \( \mathcal H \)-\( C^* \)-algebra, and let \( f\in KK^{\mathcal G}(A,\operatorname{Ind}_{\Omega}B) \). For \( i\in\{0,1\} \), let \( K_i(f)\colon K_i(A)\longrightarrow\operatorname{Ind}_{\Omega}K_i(B) \) denote the homomorphism induced by \( f \), composed with the inverse of the natural isomorphism
\[
\zeta_{\Omega,B}\colon
\operatorname{Ind}_{\Omega}K_i(B)
\longrightarrow
K_i(\operatorname{Ind}_{\Omega}B).
\]
For \( q\in\mathbb Z \), we define \( K_q(f) \) from \( K_0(f) \) and \( K_1(f) \) through the Bott periodicity isomorphisms. Then the following statements hold.
\begin{enumerate}
\item For every \( r\ge2 \), the correspondence \( \Omega \) and the class \( f \) induce functorially a morphism \( E^r_{p,q}(\mathcal G;A)\to E^r_{p,q}(\mathcal H;B) \).

\item Under the identification of the \( E^2 \)-pages with groupoid homology, the induced map is \( H_p(\Omega;K_q(f))\colon H_p(\mathcal G;K_q(A))\to H_p(\mathcal H;K_q(B)) \).

\item The induced map on the limit sheet agrees with the localisation map of the associated ABC morphism.
\end{enumerate}
\end{proposition}

\section{\texorpdfstring{The \( K \)-theory of the uniform Roe algebra}{The K-theory of the uniform Roe algebra}}
\label{sec:uniform-roe-k-theory}

% One introductory paragraph:
% "In this section, we..."

In this section, we compute the \( K \)-theory of \( C_u^*(G,\mathcal E) \). We first identify this algebra with \( \ell^\infty(G)\rtimes_rH \). We then approximate the coefficient \( C^* \)-algebra \( \ell^\infty(G) \) by separable \( H \)-invariant subalgebras and apply Proposition~\ref{prop:spectral-sequence} to the associated transformation groupoids. Next, we pass to the directed colimit, identify the \( E^2 \)-page by computing \( H_p(H;\ell^\infty(G,\mathbb Z)) \), prove that the spectral sequence degenerates integrally, and split the resulting finite filtration to obtain an explicit description of the \( K \)-groups.

\subsection{The uniform Roe algebra as a crossed product}

The following proposition is the analog of \cite[Proposition~5.1.3]{roe2003lectures} for the coarse structure determined by \( H \). We include a proof because we use this identification throughout the section.

We use the standard covariant representation of \( \ell^\infty(G)\rtimes_rH \) on \( \ell^2(G) \). For \( f\in\ell^\infty(G) \), let \( M_f \) denote multiplication by \( f \). For \( h\in H \), let \( \lambda_h \) denote the translation unitary defined by \( \lambda_h\delta_g=\delta_{g+h} \). These operators satisfy the covariance relation \( \lambda_hM_f\lambda_h^*=M_{f\cdot h} \).

\begin{proposition}
\label{prop:uniform-roe-crossed-product}
The covariant representation given by \( f\mapsto M_f \) and \( h\mapsto\lambda_h \) induces an isomorphism \( \ell^\infty(G)\rtimes_rH\cong C_u^*(G,\mathcal E) \).
\end{proposition}

\begin{proof}
Let \( \Pi:\ell^\infty(G)\rtimes_rH\to\mathcal B(\ell^2(G)) \) denote the integrated form of this covariant representation. We first show that the image of \( \Pi \) contains every controlled propagation operator.

Let \( T \) have controlled propagation. Then there exists a finite subset \( F\subseteq H \) such that \( \langle T\delta_k,\delta_g\rangle=0 \) whenever \( g-k\notin F \). For each \( h\in F \), define \( f_h(g)=\langle T\delta_{g-h},\delta_g\rangle \). Since \( |f_h(g)|\le\|T\| \) for every \( g\in G \), we have \( f_h\in\ell^\infty(G) \).

We have
\begin{equation}
\label{eq:controlled-operator-decomposition}
T=\sum_{h\in F}M_{f_h}\lambda_h.
\end{equation}
For \( g,k\in G \),
\[
\left\langle
\left(\sum_{h\in F}M_{f_h}\lambda_h\right)\delta_k,
\delta_g
\right\rangle
=
\begin{cases}
f_{g-k}(g), & g-k\in F,\\
0, & g-k\notin F.
\end{cases}
\]
If \( g-k\in F \), then \( f_{g-k}(g)=\langle T\delta_k,\delta_g\rangle \). If \( g-k\notin F \), both matrix coefficients vanish. Hence Equation~\eqref{eq:controlled-operator-decomposition} holds, so every controlled propagation operator belongs to the image of \( \Pi \).

Conversely, \( M_f\lambda_h\delta_k=f(k+h)\delta_{k+h} \), so \( M_f\lambda_h \) has controlled propagation. Hence every element of the algebraic crossed product belongs to \( C_u^*(G,\mathcal E) \), and continuity gives \( \Pi(\ell^\infty(G)\rtimes_rH)\subseteq C_u^*(G,\mathcal E) \).

It remains to prove that \( \Pi \) is faithful. We establish this by showing that the diagonal conditional expectation after applying \( \Pi \) equals the canonical conditional expectation on the reduced crossed product. We then identify \( \ell^\infty(G) \) with its multiplication representation on \( \ell^2(G) \). 

Let \( E:\ell^\infty(G)\rtimes_rH\to\ell^\infty(G) \) denote the canonical faithful conditional expectation, and let \( D:C_u^*(G,\mathcal E)\to\ell^\infty(G) \) denote the diagonal conditional expectation given by \( D(T)(g)=\langle T\delta_g,\delta_g\rangle \).

For every \( a \) in the algebraic crossed product, both \( D(\Pi(a)) \) and \( E(a) \) equal the coefficient of \( a \) at the identity element \( 0\in H \). Since both maps are continuous, we have \( D\circ\Pi=E \).

If \( a\in\ker\Pi \), then \( \Pi(a^*a)=0 \), so \( E(a^*a)=D(\Pi(a^*a))=0 \). Since \( E \) is faithful, we obtain \( a=0 \). Thus \( \Pi \) is faithful, so its image is closed. Since this image contains the dense \( ^* \)-subalgebra of controlled propagation operators, we obtain \( \Pi(\ell^\infty(G)\rtimes_rH)=C_u^*(G,\mathcal E) \).
\end{proof}

\subsection{A separable directed system}

To apply Proposition~\ref{prop:spectral-sequence} at separable stages, we construct a directed system of coefficient algebras whose colimit will be identified with \( \ell^\infty(G) \). Let \( \mathcal S \) consist of all separable, unital, \( H \)-invariant \( C^* \)-subalgebras of \( \ell^\infty(G) \) that are generated by countably many projections. We order \( \mathcal S \) by inclusion.

\begin{lemma}
\label{lem:stage-system-directed}
The partially ordered set \( \mathcal S \) is directed.
\end{lemma}

\begin{proof}
Let \( A,B\in\mathcal S \). The algebra \( C^*(A\cup B) \) is unital, \( H \)-invariant, and generated by the union of countable sets of projections that generate \( A \) and \( B \), respectively. Hence \( C^*(A\cup B)\in\mathcal S \) and contains both \( A \) and \( B \).
\end{proof}

Since every \( A\in\mathcal S \) is separable, unital, and generated by projections, its Gelfand spectrum \( \widehat{A} \) is compact, metrizable, and totally disconnected. Therefore
\begin{equation}
\label{eq:stage-coefficient-k-theory}
K_0(A)\cong C(\widehat{A},\mathbb Z),
\qquad
K_1(A)=0.
\end{equation}
The isomorphism \( K_0(A)\cong C(\widehat{A},\mathbb Z) \) is natural for pullback maps.

If \( A\subseteq B \), then the inclusion \( A\to B \) is the pullback map induced by the continuous \( H \)-equivariant surjection \( \widehat{B}\to\widehat{A} \).

\begin{lemma}
\label{lem:coefficient-colimits}
There are canonical isomorphisms
\begin{equation}
\label{eq:coefficient-stage-colimit}
\varinjlim_{A\in\mathcal S}A
\cong
\ell^\infty(G).
\end{equation}

\begin{equation}
\label{eq:integer-stage-colimit}
\varinjlim_{A\in\mathcal S}C(\widehat{A},\mathbb Z)
\cong
\ell^\infty(G,\mathbb Z).
\end{equation}
Consequently,
\[
\varinjlim_{A\in\mathcal S}K_0(A)\cong\ell^\infty(G,\mathbb Z).
\]
\end{lemma}

\begin{proof}
We first prove Equation~\eqref{eq:coefficient-stage-colimit}. Since every connecting map is an inclusion, it suffices to show that every bounded function belongs to some algebra in \( \mathcal S \).

Fix \( f\in\ell^\infty(G) \), and choose finite-range functions \( f_j \) such that \( \|f-f_j\|_\infty\to0 \). Let \( \mathcal P \) consist of the characteristic functions occurring in finite-range decompositions of the functions \( f_j \), together with all of their \( H \)-translates. Since \( H \) is countable, \( \mathcal P \) is a countable family of projections. Thus \( A=C^*(\mathcal P\cup\{1\}) \) belongs to \( \mathcal S \). Each \( f_j \) belongs to \( A \), so \( f\in A \). This proves Equation~\eqref{eq:coefficient-stage-colimit}.

We next prove Equation~\eqref{eq:integer-stage-colimit}. For each \( A\in\mathcal S \), the Gelfand transform identifies \( C(\widehat{A},\mathbb Z) \) with \( A\cap\ell^\infty(G,\mathbb Z) \). Under these identifications, the connecting maps are inclusions. Therefore the canonical homomorphism
\[
\varinjlim_{A\in\mathcal S}C(\widehat{A},\mathbb Z)
\longrightarrow
\ell^\infty(G,\mathbb Z)
\]
is injective. To prove surjectivity, it therefore suffices to show that every bounded integer-valued function belongs to some algebra in \( \mathcal S \).

Fix \( f\in\ell^\infty(G,\mathbb Z) \). Since \( f \) has finite range, write \( f=\sum_{i=1}^r n_i1_{f^{-1}(n_i)} \). Let \( \mathcal P \) consist of the projections \( 1_{f^{-1}(n_i)} \) and all of their \( H \)-translates. Then \( A=C^*(\mathcal P\cup\{1\}) \) belongs to \( \mathcal S \) and contains \( f \). Under the Gelfand transform, \( f \) therefore defines an element of \( C(\widehat A,\mathbb Z) \). This proves Equation~\eqref{eq:integer-stage-colimit}.

It remains to identify the direct limit of the groups \( K_0(A) \). The isomorphisms in Equation~\eqref{eq:stage-coefficient-k-theory} are natural with respect to the connecting maps, so they induce an isomorphism on direct limits. Combining this induced isomorphism with Equation~\eqref{eq:integer-stage-colimit} yields
\[
\varinjlim_{A\in\mathcal S}K_0(A)
\cong
\ell^\infty(G,\mathbb Z).
\]
\end{proof}

\subsection{Transformation groupoids}

For each \( A\in\mathcal S \), we apply Proposition~\ref{prop:spectral-sequence} to the transformation groupoid \( \widehat A\rtimes H \). This choice gives the relevant crossed product because \( C_r^*(\widehat A\rtimes H) \) is canonically isomorphic to \( C(\widehat A)\rtimes_rH \).

We write the right action of \( H \) on \( \widehat A \) as \( x\cdot h \), so that \( (f\cdot h)(x)=f(x\cdot(-h)) \) for every \( f\in C(\widehat A) \). The unit space of \( \widehat A\rtimes H \) is \( \widehat A \). An element \( (x,h) \) is an arrow from \( x\cdot h \) to \( x \). Thus \( s(x,h)=x\cdot h \) and \( r(x,h)=x \). Since \( s(x,h)=r(x\cdot h,k) \), the arrows \( (x,h) \) and \( (x\cdot h,k) \) are composable, and their product is \( (x,h+k) \). The inverse of \( (x,h) \) is \( (x\cdot h,-h) \).

We first record the crossed-product identification with the conventions above.

\begin{lemma}
\label{lem:transformation-groupoid-crossed-product}
For every \( A\in\mathcal S \), there is a canonical isomorphism
\[
C_r^*(\widehat A\rtimes H)
\cong 
C(\widehat A)\rtimes_rH.
\]
Under this isomorphism, \( 1_{\widehat A\times\{-h\}}f1_{\widehat A\times\{-h\}}^*=f\cdot h \) for every \( f\in C(\widehat A) \) and \( h\in H \).
\end{lemma}

\begin{proof}
The canonical isomorphism between the reduced \( C^* \)-algebra of a transformation groupoid and the corresponding reduced crossed product is standard. It remains to verify the covariance relation for the source and range maps above.

For \( h\in H \), set \( u_h=1_{\widehat A\times\{h\}} \), and regard \( f\in C(\widehat A) \) as supported on the unit space. The convolution formulas and the identity \( u_h^*=u_{-h} \) give \( (u_hfu_h^*)(x,0)=f(x\cdot h) \) and \( u_hfu_h^*=0 \) off the unit space. Hence \( u_hfu_h^*=f\cdot(-h) \), and replacing \( h \) by \( -h \) gives \( u_{-h}fu_{-h}^*=f\cdot h \).
\end{proof}

We next verify the hypotheses of Proposition~\ref{prop:spectral-sequence}. The arrow space \( \widehat A\times H \) is second-countable, locally compact, and Hausdorff. Since compact open subsets of \( \widehat A \) form a basis, the sets \( U\times\{h\} \), with \( U\subseteq\widehat A \) compact and open, form a basis of compact open bisections. Hence \( \widehat A\rtimes H \) is ample.

The isotropy group at \( x\in\widehat A \) is the stabilizer \( \{h\in H:x\cdot h=x\} \), which is torsion-free because it is a subgroup of \( H\cong\mathbb Z^n \). The groupoid is amenable because \( H \) is amenable, so it satisfies the strong Baum--Connes conjecture with coefficients by \cite[Corollary~3.15 and the subsequent paragraph]{bonicke2024categorical}. Proposition~\ref{prop:spectral-sequence} therefore applies.

\begin{proposition}
\label{prop:crossed-product-stage-colimit}
There is a canonical isomorphism
\begin{equation}
\label{eq:crossed-product-stage-colimit}
\varinjlim_{A\in\mathcal S}C_r^*(\widehat A\rtimes H)
\cong
C_u^*(G,\mathcal E).
\end{equation}
\end{proposition}

\begin{proof}
By Lemma~\ref{lem:transformation-groupoid-crossed-product}, we may identify \( C_r^*(\widehat A\rtimes H) \) with \( C(\widehat A)\rtimes_rH \) for every \( A\in\mathcal S \). For \( A\subseteq B \), the inclusion \( A\to B \), equivalently the pullback map \( C(\widehat A)\to C(\widehat B) \), is injective and \( H \)-equivariant. Since reduced crossed products commute with injective equivariant inductive limits,
\[
\varinjlim_{A\in\mathcal S}C_r^*(\widehat A\rtimes H)
\cong
\varinjlim_{A\in\mathcal S}\bigl(C(\widehat A)\rtimes_rH\bigr)
\cong
\left(\varinjlim_{A\in\mathcal S}C(\widehat A)\right)\rtimes_rH.
\]
Lemma~\ref{lem:coefficient-colimits} identifies \( \varinjlim_{A\in\mathcal S}C(\widehat A) \) with \( \ell^\infty(G) \), and Proposition~\ref{prop:uniform-roe-crossed-product} identifies \( \ell^\infty(G)\rtimes_rH \) with \( C_u^*(G,\mathcal E) \). Combining these identifications gives Equation~\eqref{eq:crossed-product-stage-colimit}.
\end{proof}

\subsection{Groupoid homology}

We identify the groupoid homology of \( \widehat A\rtimes H \) with the group homology of \( H \) with coefficients in the right \( \mathbb Z H \)-module \( C(\widehat A,\mathbb Z) \) by comparing the corresponding chain complexes.

For \( A\in\mathcal S \) and \( p\geq1 \), every composable \( p \)-tuple in \( \widehat A\rtimes H \) has the form
\[
\bigl(
(x,h_1),
(x\cdot h_1,h_2),
\ldots,
(x\cdot(h_1+\cdots+h_{p-1}),h_p)
\bigr),
\]
so we identify \( (\widehat A\rtimes H)^{(p)} \) with \( \widehat A\times H^p \). Under this identification, the face maps are
\[
d_0(x,h_1,\ldots,h_p)
=
(x\cdot h_1,h_2,\ldots,h_p),
\]
\[
d_i(x,h_1,\ldots,h_p)
=
(x,h_1,\ldots,h_i+h_{i+1},\ldots,h_p)
\qquad
(0<i<p),
\]
and
\[
d_p(x,h_1,\ldots,h_p)
=
(x,h_1,\ldots,h_{p-1}).
\]

Since \( \widehat A \) is compact and \( H \) is discrete, we identify \( C_c((\widehat A\rtimes H)^{(p)},\mathbb Z) \) with \( C(\widehat A,\mathbb Z)\otimes_{\mathbb Z}\mathbb Z[H^p] \). For \( f\in C(\widehat A,\mathbb Z) \) and \( h_1,\ldots,h_p\in H \), let \( f[h_1|\cdots|h_p] \) denote the function supported on the \( H^p \)-slice \( \{(h_1,\ldots,h_p)\} \) with coefficient function \( f \). Taking pushforwards along the face maps gives
\[
(d_0)_*\bigl(f[h_1|\cdots|h_p]\bigr)
=
(f\cdot h_1)[h_2|\cdots|h_p],
\]
\[
(d_i)_*\bigl(f[h_1|\cdots|h_p]\bigr)
=
f[h_1|\cdots|h_i+h_{i+1}|\cdots|h_p]
\qquad
(0<i<p),
\]
and
\[
(d_p)_*\bigl(f[h_1|\cdots|h_p]\bigr)
=
f[h_1|\cdots|h_{p-1}].
\]
Taking the alternating sum of these maps gives the bar differential for the right \( \mathbb Z H \)-module \( C(\widehat A,\mathbb Z) \).

\begin{lemma}
\label{lem:stage-groupoid-group-homology}
For every \( A\in\mathcal S \) and every \( p\geq0 \), there is a natural isomorphism 
\[
H_p(\widehat A\rtimes H;\mathbb Z)
\cong 
H_p(H;C(\widehat A,\mathbb Z)).
\]
If \( A\subseteq B \) and \( \pi_{BA}\colon\widehat B\to\widehat A \) denotes the canonical \( H \)-equivariant map, then in degree \( p \) pullback induces the chain map \( f[h_1|\cdots|h_p]\mapsto(\pi_{BA}^*f)[h_1|\cdots|h_p] \).
\end{lemma}

\begin{proof}
In degree zero, both chain complexes are \( C(\widehat A,\mathbb Z) \). In positive degrees, the preceding computation identifies the groupoid chain complex of \( \widehat A\rtimes H \) with the inhomogeneous bar complex for the right \( \mathbb Z H \)-module \( C(\widehat A,\mathbb Z) \). These identifications form a natural chain isomorphism, thereby inducing the asserted isomorphism in homology. If \( A\subseteq B \), then \( \pi_{BA} \) is \( H \)-equivariant, so \( \pi_{BA}^*(f\cdot h)=(\pi_{BA}^*f)\cdot h \) for every \( f\in C(\widehat A,\mathbb Z) \) and \( h\in H \). Hence pullback commutes with the bar differential.
\end{proof}

\subsection[\texorpdfstring{Group homology of \( \ell^\infty(G,\mathbb Z) \)}{Group homology of ell-infinity(G,Z)}]{Group homology of \( \ell^\infty(G,\mathbb Z) \)}
\label{sec:group-homology-of-ell-infinity(G,Z)}

To determine the direct limits of the groupoid homology groups from the preceding subsection, we compute \( H_p(H;\ell^\infty(G,\mathbb Z)) \) for every \( p\geq0 \). We first compute the top-degree homology from the Koszul complex. We then identify the Koszul chain complex with the chain complex obtained from the Koszul cochain complex by placing cohomological degree \( n-p \) in chain degree \( p \), and apply Shapiro's lemma to \( \operatorname{Map}(G,A) \). Finally, we compare integral and real coefficients, determine the dimensions of the real homology groups in degrees \( p<n \), and deduce the integral groups.

Fix an ordered basis \( e_1,\ldots,e_n \) of \( H\cong\mathbb Z^n \), and let \( e_1^*,\ldots,e_n^* \) be the dual basis of \( H^*=\operatorname{Hom}_{\mathbb Z}(H,\mathbb Z) \). For \( 1\leq i\leq n \) and \( p\geq1 \), let \( \iota_i\colon\Lambda^pH\to\Lambda^{p-1}H \) denote contraction with \( e_i^* \). We use the standard free left \( \mathbb ZH \)-resolution \( K_*\to\mathbb Z \), where \( K_p=\mathbb ZH\otimes_{\mathbb Z}\Lambda^pH \) and \( d(1\otimes\alpha)=\sum_{i=1}^n(1-e_i)\otimes\iota_i\alpha \). For a right \( \mathbb ZH \)-module \( L \), tensoring this resolution with \( L \) over \( \mathbb ZH \) gives the chain group \( L\otimes_{\mathbb Z}\Lambda^pH \) and the differential
\begin{equation}
\label{eq:translation-koszul-differential}
\partial(l\otimes\alpha)
=
\sum_{i=1}^n
(l-l\cdot e_i)\otimes\iota_i\alpha.
\end{equation}
The convention with \( l\cdot e_i-l \) multiplies the differential by \( -1 \), and multiplication by \( (-1)^p \) in degree \( p \) gives a chain isomorphism between the two conventions.

\begin{lemma}
\label{lem:top-koszul-homology}
For every right \( \mathbb ZH \)-module \( L \), there is a natural isomorphism \( H_n(H;L)\cong L^H\otimes_{\mathbb Z}\Lambda^nH \). In particular, \( H_n(H;\ell^\infty(G,\mathbb Z))\cong \ell^\infty(G/H,\mathbb Z)\otimes_{\mathbb Z}\Lambda^nH \).
\end{lemma}

\begin{proof}
When \( n=0 \), we have \( H=\{0\} \), so \( H_0(H;L)=L=L^H\otimes_{\mathbb Z}\Lambda^0H \). Assume that \( n\geq1 \). For \( l\in L^H \), let \( j_l\colon\mathbb Z\to L \) be given by \( j_l(1)=l \), where \( \mathbb Z \) carries the trivial right \( H \)-action. Since \( l\in L^H \), the map \( j_l \) is a right \( \mathbb ZH \)-module homomorphism. Functoriality of group homology defines \( \kappa_L\colon L^H\otimes_{\mathbb Z}H_n(H;\mathbb Z)\to H_n(H;L) \) by \( \kappa_L(l\otimes x)=H_n(H;j_l)(x) \). Additivity of \( l\mapsto j_l \) and of group homology in coefficient homomorphisms makes this pairing bilinear, so it factors through the tensor product. This construction does not depend on the chosen basis of \( H \).

Set \( \omega=e_1\wedge\cdots\wedge e_n \). With trivial coefficients, Equation~\eqref{eq:translation-koszul-differential} has zero differential, so the chosen Koszul resolution identifies \( H_n(H;\mathbb Z) \) with a free abelian group generated by \( [1\otimes\omega] \). For coefficients in \( L \), Equation~\eqref{eq:translation-koszul-differential} shows that \( l\otimes\omega \) is a cycle exactly when \( l=l\cdot e_i \) for every \( 1\leq i\leq n \), equivalently when \( l\in L^H \). Since there is no chain group in degree \( n+1 \), every class in \( H_n(H;L) \) has a unique representative \( l\otimes\omega \) with \( l\in L^H \). The chain map induced by \( j_l \) sends \( 1\otimes\omega \) to \( l\otimes\omega \), so \( \kappa_L(l\otimes[1\otimes\omega])=[l\otimes\omega] \). Hence \( \kappa_L \) is an isomorphism.

We next identify \( H_n(H;\mathbb Z) \) canonically. The canonical isomorphism \( H\cong H_1(H;\mathbb Z) \), together with the Pontryagin product, defines \( \tau_H\colon\Lambda^nH\to H_n(H;\mathbb Z) \) by \( h_1\wedge\cdots\wedge h_n\mapsto [h_1]\cdots[h_n] \). Graded commutativity gives \( [h][k]=-[k][h] \). For \( h\neq0 \), naturality under the inclusion \( \langle h\rangle\cong\mathbb Z\to H \) gives \( [h]^2=0 \) because \( H_2(\mathbb Z;\mathbb Z)=0 \), so the product is alternating and \( \tau_H \) is well defined. Using the chosen decomposition \( H=\mathbb Ze_1\oplus\cdots\oplus\mathbb Ze_n \), the integral Kunneth theorem gives \( H_n(H;\mathbb Z)\cong\bigotimes_{i=1}^n H_1(\mathbb Ze_i;\mathbb Z)\cong\mathbb Z \), generated by \( [e_1]\cdots[e_n] \). Since \( \tau_H(e_1\wedge\cdots\wedge e_n)=[e_1]\cdots[e_n] \), the map \( \tau_H \) is an isomorphism. Since we defined \( \tau_H \) without choosing a basis, it is canonical.

We therefore obtain the canonical composite
\[
L^H\otimes_{\mathbb Z}\Lambda^nH
\xrightarrow{\operatorname{id}\otimes\tau_H}
L^H\otimes_{\mathbb Z}H_n(H;\mathbb Z)
\xrightarrow{\kappa_L}
H_n(H;L).
\]
Its inverse gives the claimed isomorphism. If \( u\colon L\to L' \) is a right \( \mathbb ZH \)-module homomorphism, then \( u\circ j_l=j_{u(l)} \), so functoriality of group homology shows that \( \kappa_L \) is natural in \( L \). Since \( \tau_H \) depends only on \( H \), the claimed isomorphism is natural in \( L \).

Finally, the \( H \)-invariant elements of \( \ell^\infty(G,\mathbb Z) \) are precisely the bounded integer-valued functions that are constant on the cosets of \( H \). Thus \( \ell^\infty(G,\mathbb Z)^H\cong\ell^\infty(G/H,\mathbb Z) \), which gives the stated specialization.
\end{proof}

The homology groups in degrees \( p<n \) depend on both kernels and images. We therefore regard the Koszul cochain complex as a chain complex by placing cohomological degree \( n-p \) in chain degree \( p \). Since \( H \) is abelian, the underlying abelian group of every right \( \mathbb ZH \)-module \( L \) becomes a left \( \mathbb ZH \)-module through the action \( h\cdot l=l\cdot h \).

\begin{lemma}
\label{lem:koszul-self-duality}
Let \( L \) be a right \( \mathbb ZH \)-module, and equip its underlying abelian group with the left action \( h\cdot l=l\cdot h \). The chosen ordered basis of \( H \) determines a natural isomorphism \( H_p(H;L)\cong H^{n-p}(H;L) \) for every \( 0\leq p\leq n \).
\end{lemma}

\begin{proof}
Applying \( \operatorname{Hom}_{\mathbb ZH}(-,L) \) to the Koszul resolution gives 
\[ 
C^q(H;L)
\cong
\operatorname{Hom}_{\mathbb Z}(\Lambda^qH,L). 
\]
Under this identification, the coboundary satisfies
\begin{equation}
\label{eq:translation-koszul-coboundary}
(\delta\varphi)(\beta)
=
\sum_{i=1}^n
\left(
\varphi(\iota_i\beta)
-
\varphi(\iota_i\beta)\cdot e_i
\right)
\end{equation}
for every \( \varphi\in\operatorname{Hom}_{\mathbb Z}(\Lambda^qH,L) \) and \( \beta\in\Lambda^{q+1}H \).

Set \( D_p(L)=\operatorname{Hom}_{\mathbb Z}(\Lambda^{n-p}H,L) \) and let \( d_p^D=\delta^{n-p}\colon D_p(L)\to D_{p-1}(L) \). Set \( \omega=e_1\wedge\cdots\wedge e_n \). For \( \alpha\in\Lambda^pH \) and \( \beta\in\Lambda^{n-p}H \), define \( \langle\alpha,\beta\rangle\in\mathbb Z \) by \( \alpha\wedge\beta=\langle\alpha,\beta\rangle\omega \). The wedges of complementary subsets of the chosen basis give dual bases, so this pairing is perfect.

Define \( \Psi_p\colon L\otimes_{\mathbb Z}\Lambda^pH\to D_p(L) \) by
\begin{equation}
\label{eq:koszul-duality-map}
\Psi_p(l\otimes\alpha)(\beta)
=
(-1)^{p(p+3)/2}
\langle\alpha,\beta\rangle l.
\end{equation}
The perfectness of the pairing implies that each \( \Psi_p \) is an isomorphism.

We verify that \( \Psi_* \) is a chain map. Fix \( \alpha\in\Lambda^pH \) and \( \beta\in\Lambda^{n-p+1}H \). Since \( \alpha\wedge\beta\in\Lambda^{n+1}H=0 \), the contraction identity gives \( 0 = \iota_i(\alpha\wedge\beta) = \iota_i\alpha\wedge\beta + (-1)^p\alpha\wedge\iota_i\beta \). Therefore
\begin{equation}
\label{eq:contraction-pairing-sign}
\langle\alpha,\iota_i\beta\rangle
=
(-1)^{p+1}
\langle\iota_i\alpha,\beta\rangle.
\end{equation}
Equations~\eqref{eq:translation-koszul-coboundary}, \eqref{eq:koszul-duality-map}, and \eqref{eq:contraction-pairing-sign} give
\[
\begin{aligned}
d_p^D\Psi_p(l\otimes\alpha)(\beta)
&=
(-1)^{p(p+3)/2}
\sum_{i=1}^n
\langle\alpha,\iota_i\beta\rangle
(l-l\cdot e_i)\\
&=
(-1)^{p(p+3)/2+p+1}
\sum_{i=1}^n
\langle\iota_i\alpha,\beta\rangle
(l-l\cdot e_i).
\end{aligned}
\]
The exponents in the final expression and in \( \Psi_{p-1}\partial(l\otimes\alpha)(\beta) \) have the same parity because
\[
\frac{p(p+3)}2+p+1
-
\frac{(p-1)(p+2)}2
=
2p+2.
\]
Hence \( d_p^D\Psi_p=\Psi_{p-1}\partial_p \). Thus \( \Psi_* \) is a chain isomorphism from the Koszul chain complex to \( D_*(L) \), and it induces \( H_p(H;L)\cong H^{n-p}(H;L) \).

For a right \( \mathbb ZH \)-module homomorphism \( f\colon L\to L' \), Equation~\eqref{eq:koszul-duality-map} gives \( f\circ\Psi_p=\Psi_p\circ(f\otimes\operatorname{id}_{\Lambda^pH}) \). Thus, for the fixed ordered basis, the induced isomorphism is natural in \( L \).
\end{proof}

We apply Lemma~\ref{lem:koszul-self-duality} to modules of all functions.

\begin{lemma}
\label{lem:function-coefficient-homology}
Let \( A \) be an abelian group, and equip \( \operatorname{Map}(G,A) \) with the right translation action \( (f\cdot h)(g)=f(g-h) \). Then
\[
H_p\bigl(H;\operatorname{Map}(G,A)\bigr)
\cong
\begin{cases}
\operatorname{Map}(G/H,A)\otimes_{\mathbb Z}\Lambda^nH,
& p=n,\\
0,
& p\neq n.
\end{cases}
\]
\end{lemma}

\begin{proof}
The associated left action is \( (h\cdot F)(g)=F(g-h) \). We identify this module with a module coinduced from the trivial subgroup.

Choose a section \( s\colon G/H\to G \) of the quotient map. Every \( g\in G \) has a unique expression \( g=s(c)+a \), where \( c\in G/H \) and \( a\in H \). Define
\[
\Phi\colon
\operatorname{Map}(G,A)
\longrightarrow
\operatorname{Map}\bigl(H,\operatorname{Map}(G/H,A)\bigr)
\]
by \( \Phi(F)(a)(c)=F(s(c)+a) \). Its inverse satisfies \( \Phi^{-1}(\Xi)(s(c)+a)=\Xi(a)(c) \). Equip the target with the left action \( (h\cdot\Xi)(a)=\Xi(a-h) \). Then
\[
\begin{aligned}
\Phi(h\cdot F)(a)(c)
&=
(h\cdot F)(s(c)+a)\\
&=
F(s(c)+a-h)\\
&=
\Phi(F)(a-h)(c)\\
&=
(h\cdot\Phi(F))(a)(c).
\end{aligned}
\]
Thus \( \Phi \) is an isomorphism of left \( \mathbb ZH \)-modules.

Brown defines \( \operatorname{Coind}_{\{0\}}^H M =\operatorname{Hom}_{\mathbb Z}(\mathbb ZH,M) \) with its standard left \( H \)-action \cite[Chapter~III, Section~5]{brown2012cohomology}. In additive notation, identifying \( \operatorname{Hom}_{\mathbb Z}(\mathbb ZH,M) \) with the group of functions from \( H \) to \( M \), the left action is \( (h\cdot f)(a)=f(a+h) \). Precomposition by inversion \( a\mapsto-a \) gives the equivalent action \( (h\cdot\Xi)(a)=\Xi(a-h) \). Hence
\[
\operatorname{Map}(G,A)
\cong
\operatorname{Coind}_{\{0\}}^H
\operatorname{Map}(G/H,A)
\]
as left \( \mathbb ZH \)-modules.

Shapiro's lemma gives
\[
H^k\left(
H;
\operatorname{Coind}_{\{0\}}^H
\operatorname{Map}(G/H,A)
\right)
\cong
H^k\left(
\{0\};
\operatorname{Map}(G/H,A)
\right)
\]
\cite[Proposition~III.6.2]{brown2012cohomology}. The cohomology group on the right vanishes for every \( k>0 \). If \( p<n \), then \( n-p>0 \), so Lemma~\ref{lem:koszul-self-duality} gives \( H_p(H;\operatorname{Map}(G,A))=0 \). The Koszul chain groups vanish above degree \( n \), so the homology also vanishes for \( p>n \).

Lemma~\ref{lem:top-koszul-homology} gives
\[
H_n\bigl(H;\operatorname{Map}(G,A)\bigr)
\cong
\operatorname{Map}(G,A)^H
\otimes_{\mathbb Z}
\Lambda^nH.
\]
The invariant functions are precisely the functions that are constant on the cosets of \( H \). Hence \( \operatorname{Map}(G,A)^H\cong\operatorname{Map}(G/H,A) \), which proves the degree-\( n \) isomorphism.
\end{proof}

We next compare integral and real coefficients.

\begin{lemma}
\label{lem:integral-real-homology-comparison}
For every \( p<n \), the coefficient inclusion \( \ell^\infty(G,\mathbb Z)\to\ell^\infty(G,\mathbb R) \) induces an isomorphism
\[
H_p\bigl(H;\ell^\infty(G,\mathbb Z)\bigr)
\longrightarrow
H_p\bigl(H;\ell^\infty(G,\mathbb R)\bigr).
\]
\end{lemma}

\begin{proof}
The assertion is vacuous when \( n=0 \), so assume that \( n\geq1 \). Reduction modulo \( \mathbb Z \) defines a homomorphism \( \ell^\infty(G,\mathbb R)\to\operatorname{Map}(G,\mathbb R/\mathbb Z) \). Its kernel consists of the bounded integer-valued functions. Every function from \( G \) to \( \mathbb R/\mathbb Z \) has a real-valued lift with values in \( [0,1) \). Hence
\[
0
\longrightarrow
\ell^\infty(G,\mathbb Z)
\longrightarrow
\ell^\infty(G,\mathbb R)
\longrightarrow
\operatorname{Map}(G,\mathbb R/\mathbb Z)
\longrightarrow
0
\]
is a short exact sequence of right \( \mathbb ZH \)-modules.

Fix \( p<n \). The associated long exact sequence contains
\[
\begin{aligned}
H_{p+1}\bigl(H;\operatorname{Map}(G,\mathbb R/\mathbb Z)\bigr)
&\longrightarrow
H_p\bigl(H;\ell^\infty(G,\mathbb Z)\bigr)\\
&\longrightarrow
H_p\bigl(H;\ell^\infty(G,\mathbb R)\bigr)
\longrightarrow
H_p\bigl(H;\operatorname{Map}(G,\mathbb R/\mathbb Z)\bigr).
\end{aligned}
\]
For \( p<n-1 \), both outer groups vanish by Lemma~\ref{lem:function-coefficient-homology}. Hence the coefficient inclusion induces an isomorphism.

By the naturality in Lemma~\ref{lem:top-koszul-homology}, the map on \( H_n \) induced by the coefficient inclusion is identified with
\[
\ell^\infty(G/H,\mathbb R)
\otimes_{\mathbb Z}
\Lambda^nH
\longrightarrow
\operatorname{Map}(G/H,\mathbb R/\mathbb Z)
\otimes_{\mathbb Z}
\Lambda^nH,
\]
where reduction modulo \( \mathbb Z \) acts on the first tensor factor. Every function from \( G/H \) to \( \mathbb R/\mathbb Z \) has a bounded real-valued lift with values in \( [0,1) \), so the map on the first factor is surjective. Since \( \Lambda^nH \) is free of rank one, the displayed map is surjective. Exactness forces the connecting homomorphism to vanish. Therefore, the coefficient inclusion also induces an isomorphism in degree \( n-1 \).
\end{proof}

For \( p<n \), Lemma~\ref{lem:integral-real-homology-comparison} identifies the integral homology group, as an abelian group, with the additive group of a real vector space. We now determine its real dimension. Let \( \mathfrak c=2^{\aleph_0} \). We obtain the lower bound by constructing cycles detected by translation-invariant means.

A translation-invariant mean on a group \( Q \) is a positive normalized linear functional on \( \ell^\infty(Q,\mathbb R) \) that is invariant under the right translation action.

\begin{lemma}
\label{lem:separating-invariant-means}
Let \( Q\cong\mathbb Z^d \) with \( d\geq1 \). There exist subsets \( T_\alpha\subseteq Q \) and translation-invariant means
\[
\mu_\alpha\colon
\ell^\infty(Q,\mathbb R)
\longrightarrow
\mathbb R,
\qquad
\alpha\in\{0,1\}^{\mathbb N},
\]
such that
\begin{equation}
\label{eq:separating-means}
\mu_\alpha(1_{T_\beta})
=
\begin{cases}
1,
& \alpha=\beta,\\
0,
& \alpha\neq\beta.
\end{cases}
\end{equation}
\end{lemma}

\begin{proof}
Let \( \{0,1\}^{<\mathbb N} \) denote the set of finite binary words, and choose an injective map \( \eta\colon\{0,1\}^{<\mathbb N}\to\mathbb N \). For \( \alpha\in\{0,1\}^{\mathbb N} \), let \( \alpha|_r \) denote its initial word of length \( r \), and define \( A_\alpha=\{\eta(\alpha|_r):r\geq1\} \). Each \( A_\alpha \) is infinite. Distinct binary sequences have only finitely many common initial words, so \( A_\alpha\cap A_\beta \) is finite whenever \( \alpha\neq\beta \).

Fix an identification \( Q\cong\mathbb Z^d \), and let \( B_k=\{-k,\ldots,k\}^d \) for \( k\geq1 \). We construct pairwise disjoint translates \( C_k=a_k+B_k \) inductively. Suppose that we have chosen \( C_1,\ldots,C_{k-1} \). The set
\[
R_k
=
\bigcup_{j<k}
\{x-y:x\in C_j,\ y\in B_k\}
\]
is finite. Choose \( a_k\in Q\setminus R_k \). If \( x\in C_k\cap C_j \) for some \( j<k \), then \( x=a_k+y \) for some \( y\in B_k \), so \( a_k=x-y\in R_k \), a contradiction. Hence the sets \( C_k \) are pairwise disjoint.

Fix \( q=(q_1,\ldots,q_d)\in Q \). For each \( i \), translation by \( q_i \) changes at most \( 2|q_i| \) points in every line parallel to the \( i \)-th coordinate. Summing over the coordinates gives
\[
|(B_k+q)\mathbin{\triangle}B_k|
\leq
2\sum_{i=1}^d
|q_i|(2k+1)^{d-1}.
\]
Dividing by \( |C_k|=|B_k|=(2k+1)^d \) gives
\begin{equation}
\label{eq:folner-cube-estimate}
\frac{|(C_k+q)\mathbin{\triangle}C_k|}{|C_k|}
\longrightarrow
0.
\end{equation}

For \( \alpha\in\{0,1\}^{\mathbb N} \), define \( T_\alpha=\bigcup_{k\in A_\alpha}C_k \). Since the sets \( C_k \) are pairwise disjoint,
\[
T_\alpha\cap T_\beta
=
\bigcup_{k\in A_\alpha\cap A_\beta}C_k.
\]
When \( \alpha\neq\beta \), the index set on the right is finite, and each \( C_k \) is finite. Hence \( T_\alpha\cap T_\beta \) is finite.

Fix a free ultrafilter \( \mathcal U \) on \( \mathbb N \). Every bounded real sequence is contained in a compact interval and therefore has a unique \( \mathcal U \)-limit. Continuity of addition and scalar multiplication implies that ultrafilter limits preserve these operations, and order preservation gives positivity.

For each \( \alpha \), enumerate \( A_\alpha=\{k_{\alpha,j}:j\geq1\} \) in increasing order, and define
\[
\mu_\alpha(\varphi)
=
\lim_{j\to\mathcal U}
\frac{1}{|C_{k_{\alpha,j}}|}
\sum_{x\in C_{k_{\alpha,j}}}
\varphi(x),
\qquad
\varphi\in\ell^\infty(Q,\mathbb R).
\]
The functional \( \mu_\alpha \) is linear, positive, and normalized.

We show that \( \mu_\alpha \) is translation-invariant. For \( q\in Q \) and \( \varphi\in\ell^\infty(Q,\mathbb R) \), the right action satisfies \( (\varphi\cdot q)(x)=\varphi(x-q) \). Hence
\[
\begin{aligned}
\left|
\frac{1}{|C_k|}
\sum_{x\in C_k}
(\varphi\cdot q)(x)
-
\frac{1}{|C_k|}
\sum_{x\in C_k}
\varphi(x)
\right|
&=
\left|
\frac{1}{|C_k|}
\sum_{x\in C_k-q}
\varphi(x)
-
\frac{1}{|C_k|}
\sum_{x\in C_k}
\varphi(x)
\right|\\
&\leq
\|\varphi\|_\infty
\frac{|(C_k-q)\mathbin{\triangle}C_k|}{|C_k|}.
\end{aligned}
\]
Equation~\eqref{eq:folner-cube-estimate}, applied to \( -q \), shows that the right-hand side converges to zero. The same convergence holds along the subsequence \( k=k_{\alpha,j} \). Therefore \( \mu_\alpha(\varphi\cdot q)=\mu_\alpha(\varphi) \).

Every set \( C_{k_{\alpha,j}} \) lies in \( T_\alpha \), so \( \mu_\alpha(1_{T_\alpha})=1 \). If \( \alpha\neq\beta \), then \( A_\alpha\cap A_\beta \) is finite. Hence \( C_{k_{\alpha,j}}\cap T_\beta=\varnothing \) for all sufficiently large \( j \), so \( \mu_\alpha(1_{T_\beta})=0 \). This proves Equation~\eqref{eq:separating-means}.
\end{proof}

We use these means to construct independent homology classes.

\begin{lemma}
\label{lem:independent-lower-koszul-classes}
For every \( 0\leq p<n \), the real vector space \( H_p(H;\ell^\infty(H,\mathbb R)) \) contains a linearly independent family of cardinality \( \mathfrak c \).
\end{lemma}

\begin{proof}
Fix \( p<n \). Let \( P=\langle e_1,\ldots,e_p\rangle \) when \( p>0 \), and let \( P=\{0\} \) when \( p=0 \). Let \( Q=\langle e_{p+1},\ldots,e_n\rangle \). Then \( H=P\oplus Q \) and \( Q\cong\mathbb Z^{n-p} \). Set \( \omega_p=e_1\wedge\cdots\wedge e_p \) when \( p>0 \), and set \( \omega_0=1 \).

By Lemma~\ref{lem:separating-invariant-means}, choose subsets \( T_\alpha\subseteq Q \) together with translation-invariant means \( \mu_\alpha\colon\ell^\infty(Q,\mathbb R)\to\mathbb R \) that satisfy Equation~\eqref{eq:separating-means}. Then choose a translation-invariant mean \( \nu\colon\ell^\infty(P,\mathbb R)\to\mathbb R \). When \( p=0 \), take evaluation at \( 0 \). When \( p>0 \), Lemma~\ref{lem:separating-invariant-means} provides such a mean.

For \( F\in\ell^\infty(H,\mathbb R) \), define
\[
M_\alpha(F)
=
\mu_\alpha\left(
y
\longmapsto
\nu\bigl(
x\longmapsto F(x+y)
\bigr)
\right).
\]
For every \( y\in Q \), the inner value has absolute value at most \( \|F\|_\infty \), so the resulting function belongs to \( \ell^\infty(Q,\mathbb R) \).

We verify that \( M_\alpha \) is invariant under the right \( H \)-action. If \( x_0\in P \), then \( (F\cdot x_0)(x+y)=F((x-x_0)+y) \), so translation invariance of \( \nu \) gives \( M_\alpha(F\cdot x_0)=M_\alpha(F) \). If \( y_0\in Q \), then \( (F\cdot y_0)(x+y)=F(x+(y-y_0)) \), so translation invariance of \( \mu_\alpha \) gives \( M_\alpha(F\cdot y_0)=M_\alpha(F) \). Since every \( h\in H \) has a unique form \( h=x_0+y_0 \), we obtain \( M_\alpha(F\cdot h)=M_\alpha(F) \). Since \( \nu \) and \( \mu_\alpha \) are linear, \( M_\alpha\colon\ell^\infty(H,\mathbb R)\to\mathbb R \) is a homomorphism of right \( \mathbb ZH \)-modules, where \( \mathbb R \) carries the trivial action.

For \( \beta\in\{0,1\}^{\mathbb N} \), define \( f_\beta(x+y)=1_{T_\beta}(y) \) for \( x\in P \) and \( y\in Q \). The decomposition \( H=P\oplus Q \) makes this formula unambiguous. For \( 1\leq i\leq p \), the function \( f_\beta \) is invariant under \( e_i \), so \( f_\beta-f_\beta\cdot e_i=0 \). For \( p<i\leq n \), we have \( \iota_i\omega_p=0 \). Hence every summand in Equation~\eqref{eq:translation-koszul-differential} vanishes, so \( z_\beta=f_\beta\otimes\omega_p \) is a cycle.

For every \( y\in Q \), \( \nu(x\mapsto f_\beta(x+y))=1_{T_\beta}(y) \). Equation~\eqref{eq:separating-means} gives
\[
M_\alpha(f_\beta)
=
\begin{cases}
1,
& \alpha=\beta,\\
0,
& \alpha\neq\beta.
\end{cases}
\]

The coefficient homomorphism \( M_\alpha \) induces the chain map \( M_\alpha\otimes\operatorname{id}_{\Lambda^*H} \) from the Koszul complex with coefficients in \( \ell^\infty(H,\mathbb R) \) to the Koszul complex with trivial coefficients \( \mathbb R \). Indeed, \( M_\alpha(l-l\cdot e_i)=0 \) for every \( l\in\ell^\infty(H,\mathbb R) \) and every \( 1\leq i\leq n \). The chain map sends \( z_\beta \) to \( M_\alpha(f_\beta)(1\otimes\omega_p) \). The Koszul differential with trivial coefficients is zero, so \( H_p(H;\mathbb R)\cong\mathbb R\otimes_{\mathbb Z}\Lambda^pH \), and \( [1\otimes\omega_p]\neq0 \).

Let \( I\subseteq\{0,1\}^{\mathbb N} \) be finite, and suppose that
\[
\sum_{\beta\in I}
a_\beta[z_\beta]
=
0,
\qquad
a_\beta\in\mathbb R.
\]
Fix \( \alpha\in I \). The homology homomorphism induced by \( M_\alpha \) gives \( a_\alpha[1\otimes\omega_p]=0 \). Hence \( a_\alpha=0 \). Since \( \alpha \) was arbitrary, every coefficient vanishes. Since \( \lvert\{0,1\}^{\mathbb N}\rvert=\mathfrak c \), the classes \( [z_\beta] \) form a linearly independent family of cardinality \( \mathfrak c \).
\end{proof}

We now determine the dimension of the real homology groups in degrees \( p<n \).

\begin{theorem}
\label{thm:lower-real-homology-dimension}
For every \( 0\leq p<n \),
\[
\dim_{\mathbb R}
H_p\bigl(H;\ell^\infty(G,\mathbb R)\bigr)
=
\mathfrak c.
\]
\end{theorem}

\begin{proof}
Fix \( p<n \). Since \( G \) is countable, \( |\ell^\infty(G,\mathbb R)|=\mathfrak c \). Indeed, \( \ell^\infty(G,\mathbb R) \) embeds into \( \mathbb R^{\mathbb N} \), which has cardinality \( \mathfrak c \), and the constant functions give the reverse inequality.

The degree-\( p \) Koszul chain group is a finite direct sum of copies of \( \ell^\infty(G,\mathbb R) \), so it has cardinality \( \mathfrak c \). Its homology therefore has cardinality at most \( \mathfrak c \), and \( \dim_{\mathbb R}H_p(H;\ell^\infty(G,\mathbb R))\leq\mathfrak c \).

Choose a coset \( O=g_0+H \). Since \( O \) is \( H \)-invariant, extension by zero and restriction define right \( \mathbb ZH \)-module homomorphisms
\[
\operatorname{ext}_O\colon
\ell^\infty(O,\mathbb R)
\longrightarrow
\ell^\infty(G,\mathbb R)
\]
and
\[
\operatorname{res}_O\colon
\ell^\infty(G,\mathbb R)
\longrightarrow
\ell^\infty(O,\mathbb R).
\]
They satisfy \( \operatorname{res}_O\circ\operatorname{ext}_O=\operatorname{id} \), so the induced homology map is split injective.

Translation by \( -g_0 \) gives an isomorphism of right \( H \)-sets \( O\cong H \), and therefore an isomorphism of the corresponding coefficient modules. 

Lemma~\ref{lem:independent-lower-koszul-classes} gives \( \dim_{\mathbb R}H_p(H;\ell^\infty(O,\mathbb R))\geq\mathfrak c \). Since the induced homology map is split injective, it follows that \( \dim_{\mathbb R}H_p(H;\ell^\infty(G,\mathbb R))\geq\mathfrak c \). The two bounds prove the theorem.
\end{proof}

We now return to integral coefficients. Let \( \mathbb Q^{(\mathfrak c)} \) denote the direct sum of \( \mathfrak c \) copies of \( \mathbb Q \).

\begin{theorem}
\label{thm:integer-coefficient-group-homology}
For every \( p\geq0 \),
\[
H_p\bigl(H;\ell^\infty(G,\mathbb Z)\bigr)
\cong
\begin{cases}
\mathbb Q^{(\mathfrak c)},
& 0\leq p<n,\\
\ell^\infty(G/H,\mathbb Z)
\otimes_{\mathbb Z}
\Lambda^nH,
& p=n,\\
0,
& p>n.
\end{cases}
\]
The isomorphisms for \( 0\leq p<n \) are noncanonical. The chosen ordered basis of \( H \) identifies the degree-\( n \) group with \( \ell^\infty(G/H,\mathbb Z) \).
\end{theorem}

\begin{proof}
Fix \( p<n \), and set \( V=H_p(H;\ell^\infty(G,\mathbb R)) \). Lemma~\ref{lem:integral-real-homology-comparison} identifies \( H_p(H;\ell^\infty(G,\mathbb Z)) \), as an abelian group, with the additive group underlying \( V \). Theorem~\ref{thm:lower-real-homology-dimension} gives \( \dim_{\mathbb R}V=\mathfrak c \).

We claim that \( \dim_{\mathbb Q}\mathbb R=\mathfrak c \). A \( \mathbb Q \)-basis of \( \mathbb R \) is a subset of \( \mathbb R \), so its cardinality is at most \( \mathfrak c \). If its cardinality is \( \kappa \), then every real number is a finite rational linear combination of basis elements. Since \( \mathbb Q \) is countable, \( |\mathbb R| = \max(\aleph_0,\kappa) \). Because \( |\mathbb R|=\mathfrak c>\aleph_0 \), we obtain \( \kappa=\mathfrak c \). Therefore
\[
\dim_{\mathbb Q}V
=
\dim_{\mathbb R}V
\cdot
\dim_{\mathbb Q}\mathbb R
=
\mathfrak c.
\]
Hence the additive group underlying \( V \) is a \( \mathbb Q \)-vector space of dimension \( \mathfrak c \), so
\[
H_p\bigl(H;\ell^\infty(G,\mathbb Z)\bigr)
\cong
\mathbb Q^{(\mathfrak c)}.
\]
This isomorphism is noncanonical because it requires a choice of a \( \mathbb Q \)-basis of \( V \).

Lemma~\ref{lem:top-koszul-homology} gives the natural isomorphism
\[
H_n\bigl(H;\ell^\infty(G,\mathbb Z)\bigr)
\cong
\ell^\infty(G/H,\mathbb Z)
\otimes_{\mathbb Z}
\Lambda^nH.
\]
The chosen ordered basis identifies \( \Lambda^nH \) with \( \mathbb Z \), and therefore identifies the degree-\( n \) group with \( \ell^\infty(G/H,\mathbb Z) \). Since the Koszul chain groups vanish above degree \( n \), the homology vanishes for \( p>n \).
\end{proof}

\subsection{The spectral sequences for the transformation groupoids}

For every \( A\in\mathcal S \), Proposition~\ref{prop:spectral-sequence} and Equation~\eqref{eq:groupoid-homology-spectral-sequence} give a convergent spectral sequence
\begin{equation}
\label{eq:stage-spectral-sequence}
E^2_{p,q}(A)
\cong
H_p\bigl(\widehat A\rtimes H;K_q(\mathbb C)\bigr)
\Longrightarrow
K_{p+q}\bigl(C_r^*(\widehat A\rtimes H)\bigr).
\end{equation}

Bott periodicity and Lemma~\ref{lem:stage-groupoid-group-homology} identify its \( E^2 \)-page:
\begin{equation}
\label{eq:stage-second-page}
E^2_{p,q}(A)
\cong
\begin{cases}
H_p\bigl(H;C(\widehat A,\mathbb Z)\bigr), & q \text{ is even},\\
0, & q \text{ is odd}.
\end{cases}
\end{equation}

Since the Koszul resolution of the trivial \( \mathbb ZH \)-module has length \( n \), the group homology groups in Equation~\eqref{eq:stage-second-page} are zero for \( p>n \). Therefore
\begin{equation}
\label{eq:stage-horizontal-bound}
E^2_{p,q}(A)=0
\qquad
(p>n).
\end{equation}

\subsection{Functoriality of the spectral sequences}

We next show that the spectral sequences associated to the transformation groupoids \( \widehat A\rtimes H \) are compatible with the inclusions in the directed system \( \mathcal S \). For \( A\subseteq B \), the factor map determined by \( \pi_{BA} \) induces a morphism of spectral sequences by Proposition~\ref{prop:miller-functoriality}. We identify the induced maps on the second page and on the abutment and then verify compatibility with composition.

Let \( A\subseteq B \). Define \( \varphi_{BA}\colon\widehat B\rtimes H\to\widehat A\rtimes H \) by \( \varphi_{BA}(x,h)=(\pi_{BA}(x),h) \). Since \( \pi_{BA} \) is continuous, surjective, and \( H \)-equivariant, \( \varphi_{BA} \) is a continuous surjective groupoid homomorphism. We verify that \( \varphi_{BA} \) is a factor map. Let \( K\subseteq\widehat A\times H \) be compact. Since \( H \) is discrete, the projection of \( K \) onto \( H \) is finite. For each \( h \) in this projection, the set \( K_h=\{x\in\widehat A:(x,h)\in K\} \) is compact. Hence
\[
\varphi_{BA}^{-1}(K)
=
\bigcup_h
\pi_{BA}^{-1}(K_h)\times\{h\}
\]
is a finite union of compact sets because \( \pi_{BA} \) is proper. Thus \( \varphi_{BA} \) is proper. For each \( x\in\widehat B \), the source fiber at \( x \) consists of the arrows \( (x\cdot(-h),h) \) with \( h\in H \). The restriction of \( \varphi_{BA} \) sends the arrow with \( H \)-coordinate \( h \) to the arrow with the same \( H \)-coordinate in the source fiber at \( \pi_{BA}(x) \). Hence \( \varphi_{BA} \) is bijective on source fibers. Therefore \( \varphi_{BA} \) is a factor map in the sense of \cite[Example~3.6]{miller2023functors}.

The factor map induces the associated action correspondence \( \Omega_{BA} \) from \( \widehat A\rtimes H \) to \( \widehat B\rtimes H \). Its action anchor is \( \pi_{BA}\colon\widehat B\to\widehat A \), which is proper because \( \widehat B \) is compact and \( \widehat A \) is Hausdorff. Hence \( \Omega_{BA} \) is proper by \cite[Example~3.5]{miller2023functors}. By \cite[Example~5.1]{miller2025isomorphisms}, the proper correspondence determines the canonical morphism
\[
f_{BA}=f_{\Omega_{BA}}
\in
KK^{\widehat A\rtimes H}
\left(
C(\widehat A),
\operatorname{Ind}_{\Omega_{BA}}C(\widehat B)
\right).
\]
The trivial families \( \{\widehat A\} \) and \( \{\widehat B\} \) are compatible with \( \Omega_{BA} \) by \cite[Proposition~3.22]{miller2025isomorphisms}. Proposition~\ref{prop:miller-functoriality}, applied to \( \Omega_{BA} \) and \( f_{BA} \), therefore gives a morphism
\begin{equation}
\label{eq:stage-spectral-sequence-map}
m_{BA}^r:
E^r_{p,q}(A)
\longrightarrow
E^r_{p,q}(B).
\end{equation}

\begin{proposition}
\label{prop:stage-map-identification}
Let \( A\subseteq B \).

The morphism \( m_{BA}^2 \) induces
\[
H_p(H;\pi_{BA}^*):
H_p\bigl(H;C(\widehat A,\mathbb Z)\bigr)
\longrightarrow
H_p\bigl(H;C(\widehat B,\mathbb Z)\bigr)
\]
in every even row of the \( E^2 \)-page. The induced map on the abutment is
\[
K_*(\pi_{BA}^*\rtimes_rH):
K_*\bigl(C_r^*(\widehat A\rtimes H)\bigr)
\longrightarrow
K_*\bigl(C_r^*(\widehat B\rtimes H)\bigr).
\]
\end{proposition}

\begin{proof}
We first consider the map on the second page. Under the identification \( (\widehat A\rtimes H)^{(p)}\cong\widehat A\times H^p \), the map induced by \( \varphi_{BA} \) on composable \( p \)-tuples is
\[
(x,h_1,\ldots,h_p)
\longmapsto
(\pi_{BA}(x),h_1,\ldots,h_p).
\]
Its pullback sends \( f[h_1|\cdots|h_p] \) to \( (\pi_{BA}^*f)[h_1|\cdots|h_p] \). Example~3.9 of \cite{miller2025ample} identifies the correspondence-induced map on groupoid homology with this pullback. Proposition~\ref{prop:miller-functoriality} therefore identifies this map with \( m_{BA}^2 \).

We next identify the induced map on the abutment. By Proposition~\ref{prop:miller-functoriality}, the map on the limit sheet is the localisation map of the ABC morphism associated to \( (\Omega_{BA},f_{BA}) \). Since \( \widehat A\rtimes H \) and \( \widehat B\rtimes H \) have torsion-free stabilizers, the families \( \{\widehat A\} \) and \( \{\widehat B\} \) satisfy condition (P) by \cite[Example~3.18]{miller2025isomorphisms}. By \cite[Proposition~3.17]{miller2025isomorphisms}, these families yield the proper complementary pairs. Hence \cite[Example~5.11]{miller2025isomorphisms} identifies the localisation map with \( K_*^{\mathrm{top}}(\Omega_{BA};f_{BA}) \).

Since \( \widehat A\rtimes H \) is amenable, the factorisation hypothesis in \cite[Example~5.11]{miller2025isomorphisms} holds. Hence \( K_*^{\mathrm{red}}(\Omega_{BA};f_{BA}) \) is defined. Since \( \widehat A\rtimes H \) and \( \widehat B\rtimes H \) satisfy the strong Baum--Connes conjecture with coefficients, they satisfy the Baum--Connes conjecture with the coefficients required in \cite[Example~5.11]{miller2025isomorphisms}. Therefore \( K_*^{\mathrm{top}}(\Omega_{BA};f_{BA}) = K_*^{\mathrm{red}}(\Omega_{BA};f_{BA}) \).

For the canonical class \( f_{BA}=f_{\Omega_{BA}} \), Proposition~5.4 and Definition~6.13 of \cite{miller2023functors} identify the corresponding full map with \( K_*(C^*(\Omega_{BA})) \). Using amenability and \cite[Example~3.6]{miller2023functors}, the reduced counterpart is the \( K \)-theory map of the inclusion of reduced groupoid \( C^* \)-algebras induced by \( \varphi_{BA} \). The inclusion induced by \( \varphi_{BA} \) sends \( a\in C_c(\widehat A\rtimes H) \) to \( a\circ\varphi_{BA}\in C_c(\widehat B\rtimes H) \), where \( (a\circ\varphi_{BA})(x,h)=a(\pi_{BA}(x),h) \). Under the canonical isomorphisms of Lemma~\ref{lem:transformation-groupoid-crossed-product}, the inclusion induced by \( \varphi_{BA} \) is \( \pi_{BA}^*\rtimes_r H \). Hence the induced map on the abutment is \( K_*(\pi_{BA}^*\rtimes_r H) \).
\end{proof}

It remains to show that the morphisms \( m_{BA}^r \) make the spectral sequences indexed by \( A\in\mathcal S \) into a directed system. Let \( A\subseteq B\subseteq C \). We have \( \varphi_{BA}\circ\varphi_{CB}=\varphi_{CA} \). The action correspondences \( \Omega_{BA} \) and \( \Omega_{CB} \) have correspondence spaces \( \widehat B\rtimes H \) and \( \widehat C\rtimes H \), respectively. Their composite is the balanced product \( ((\widehat B\rtimes H)\times_{\widehat B}(\widehat C\rtimes H))/(\widehat B\rtimes H) \), where the fiber product is taken with respect to \( s \) on \( \widehat B\rtimes H \) and \( \pi_{CB}\circ r \) on \( \widehat C\rtimes H \). The map \( (\omega,\lambda)\mapsto\omega\cdot\lambda \), where \( \omega\in\widehat B\rtimes H \) acts on \( \lambda\in\widehat C\rtimes H \) through the left action associated to \( \varphi_{CB} \), is constant on the balancing relation by associativity of this action. It therefore descends to a map \( [\omega,\lambda]\mapsto\omega\cdot\lambda \) on the balanced product. Its inverse sends \( \lambda \) to \( [1_{\pi_{CB}(r(\lambda))},\lambda] \), so the descended map is an equivariant homeomorphism onto the correspondence space \( \widehat C\rtimes H \) of \( \Omega_{CA} \). Under this homeomorphism, the left \( \widehat A\rtimes H \)-action is the action associated to \( \varphi_{BA}\circ\varphi_{CB}=\varphi_{CA} \), and the right \( \widehat C\rtimes H \)-action is the usual right multiplication action. Hence \( \Omega_{CB}\circ\Omega_{BA}\cong\Omega_{CA} \).

By \cite[Remark~5.5]{miller2025isomorphisms}, the canonical triples associated to \( (\Omega_{BA},f_{BA}) \) and \( (\Omega_{CB},f_{CB}) \) compose, up to equivalence, to the canonical triple associated to \( (\Omega_{CA},f_{CA}) \). By \cite[Example~5.8]{miller2025isomorphisms}, these canonical triples determine the corresponding ABC morphisms. The functoriality in \cite[Theorem~6.1]{miller2025isomorphisms} therefore gives \( m_{CB}^r\circ m_{BA}^r=m_{CA}^r \) for every \( r\ge2 \). For \( A\in\mathcal S \), we have \( \pi_{AA}=\operatorname{id}_{\widehat A} \) and \( \varphi_{AA}=\operatorname{id}_{\widehat A\rtimes H} \), so the construction gives the identity correspondence. The functoriality in \cite[Theorem~6.1]{miller2025isomorphisms} therefore gives \( m_{AA}^r=\operatorname{id} \) for every \( r\ge2 \). Hence the spectral sequences form a directed system under the morphisms \( m_{BA}^r \). By Proposition~\ref{prop:stage-map-identification}, the induced maps on the second page are \( H_p(H;\pi_{BA}^*) \), and the induced maps on the abutment are \( K_*(\pi_{BA}^*\rtimes_r H) \).

\subsection{Rational degeneration}

Fix \( A\in\mathcal S \). For the transformation groupoid \( \widehat A\rtimes H \), the Proietti--Yamashita spectral sequence in Equation~(0.1) is the spectral sequence in Equation~\eqref{eq:stage-spectral-sequence}. Proietti and Yamashita state that Equation~(0.1) is rationally degenerate at the \( E^2 \)-sheet \cite[p.~2, Equation~(0.1) and the paragraph following it]{proietti2025chern}. Thus, for every \( r\geq2 \),
\begin{equation}
\label{eq:stage-rational-degeneration}
d_r(A)\otimes\operatorname{id}_{\mathbb Q}=0.
\end{equation}

Since \( A\in\mathcal S \) was arbitrary, Equation~\eqref{eq:stage-rational-degeneration} holds for every \( A\in\mathcal S \). We use this vanishing in the following subsection to prove that the higher differentials of the directed-colimit spectral sequence vanish after rationalization.

\subsection{The spectral sequence after taking directed colimits}
\label{sec:the-spectral-sequence-after-taking-directed-colimits}

We take the directed colimit of the spectral sequences associated to \( A\in\mathcal S \). Proposition~\ref{prop:crossed-product-stage-colimit} identifies the directed colimit of the \( C^* \)-algebras whose \( K \)-theory groups form the abutments with \( C_u^*(G,\mathcal E) \).

For every \( r\ge2 \), let
\[
E^r_{p,q}
=
\varinjlim_{A\in\mathcal S}E^r_{p,q}(A).
\]
Since the maps in Equation~\eqref{eq:stage-spectral-sequence-map} form morphisms of spectral sequences, they induce differentials
\[
d_r\colon
E^r_{p,q}
\longrightarrow
E^r_{p-r,q+r-1}.
\]
Filtered colimits are exact in the category of abelian groups. They therefore commute with the kernels and images defining the next page. Hence
\[
E^{r+1}_{p,q}
\cong
\varinjlim_{A\in\mathcal S}E^{r+1}_{p,q}(A),
\]
so the groups \( E^r_{p,q} \) form a spectral sequence.

Exactness of filtered colimits and the finite free Koszul resolution imply that group homology commutes with the colimit over \( \mathcal S \). Lemma~\ref{lem:coefficient-colimits} and Equation~\eqref{eq:stage-second-page} therefore give
\begin{equation}
\label{eq:colimit-second-page}
E^2_{p,q}
\cong
\begin{cases}
H_p\bigl(H;\ell^\infty(G,\mathbb Z)\bigr), & q \text{ even},\\
0, & q \text{ odd},
\end{cases}
\end{equation}
and
\begin{equation}
\label{eq:colimit-horizontal-bound}
E^2_{p,q}=0
\qquad
(p>n),
\end{equation}
because the Koszul resolution has length \( n \).

Since tensor products commute with filtered colimits, Equation~\eqref{eq:stage-rational-degeneration} gives
\begin{equation}
\label{eq:colimit-rational-degeneration}
d_r\otimes_{\mathbb Z}\mathbb Q=0
\qquad
(r\ge2).
\end{equation}
Therefore every element of \( \operatorname{im}(d_r) \) is torsion.

The remaining point is to identify the abutment of the colimit spectral sequence.

\begin{proposition}
\label{prop:directed-colimit-convergence}
The spectral sequence \( E^r_{p,q} \) strongly converges to \( K_{p+q}(C_u^*(G,\mathcal E)) \).
\end{proposition}

\begin{proof}
For every \( A\in\mathcal S \), \cite[Theorem~2.19]{miller2025isomorphisms} gives strong convergence of \( E^r_{p,q}(A) \) to the corresponding ABC localisation. Since \( \widehat A\rtimes H \) satisfies the strong Baum--Connes conjecture with coefficients, this localisation identifies with \( K_*(C_r^*(\widehat A\rtimes H)) \). Thus, for every \( N\in\mathbb Z \), the spectral sequence associated to \( A \) determines an exhaustive separated filtration
\[
\cdots
\subseteq
F_{p-1}^AK_N
\subseteq
F_p^AK_N
\subseteq
F_{p+1}^AK_N
\subseteq
\cdots
\]
of \( K_N(C_r^*(\widehat A\rtimes H)) \) with
\[
F_p^AK_N/F_{p-1}^AK_N
\cong
E^\infty_{p,N-p}(A).
\]

The spectral sequence associated to \( A \) is concentrated in degrees \( 0\leq p\leq n \). Hence \( E^\infty_{p,q}(A)=0 \) for \( p<0 \) and \( p>n \). Since the filtration is separated and exhaustive, we obtain \( F_p^AK_N=0 \) for \( p<0 \) and \( F_p^AK_N=K_N(C_r^*(\widehat A\rtimes H)) \) for \( p\ge n \). By Proposition~\ref{prop:stage-map-identification}, the abutment map associated to \( A\subseteq B \) is \( K_N(\pi_{BA}^*\rtimes_rH) \). By \cite[Lemma~6.8]{miller2025isomorphisms}, the corresponding localisation map preserves these filtrations. Hence \( K_N(\pi_{BA}^*\rtimes_rH) \) maps \( F_p^AK_N \) into \( F_p^BK_N \).

For \( 0\le p\le n \), let
\[
F_pK_N
=
\varinjlim_{A\in\mathcal S}F_p^AK_N.
\]
Exactness of filtered colimits gives the finite exhaustive filtration
\[
0
=
F_{-1}K_N
\subseteq
F_0K_N
\subseteq
\cdots
\subseteq
F_nK_N
=
\varinjlim_{A\in\mathcal S}
K_N(C_r^*(\widehat A\rtimes H)).
\]
Continuity of \( K \)-theory and Proposition~\ref{prop:crossed-product-stage-colimit} identify the last term with
\begin{equation}
\label{eq:colimit-abutment-group}
F_nK_N
\cong
K_N(C_u^*(G,\mathcal E)).
\end{equation}

Let \( R=\max\{2,n+1\} \). Since \( E^r_{p,q}(A)=0 \) for \( p<0 \) and \( p>n \), every differential \( d_r(A) \) vanishes for \( r>n \). Hence
\[
E^R_{p,q}(A)
=
E^\infty_{p,q}(A)
\]
for every \( A\in\mathcal S \). Therefore
\[
E^R_{p,q}
=
\varinjlim_{A\in\mathcal S}E^R_{p,q}(A)
=
\varinjlim_{A\in\mathcal S}E^\infty_{p,q}(A)
=
E^\infty_{p,q}.
\]

Finally,
\[
\begin{aligned}
F_pK_N/F_{p-1}K_N
&\cong
\varinjlim_{A\in\mathcal S}
\left(
F_p^AK_N/F_{p-1}^AK_N
\right)\\
&\cong
\varinjlim_{A\in\mathcal S}
E^\infty_{p,N-p}(A)\\
&\cong
E^\infty_{p,N-p},
\end{aligned}
\]
where the first isomorphism follows from the exactness of filtered colimits. Therefore the filtration of \( K_N(C_u^*(G,\mathcal E)) \) is finite, exhaustive, and separated, and its associated graded groups are \( E^\infty_{p,N-p} \). Hence the spectral sequence strongly converges.
\end{proof}

\subsection{Integral degeneration}

To identify the associated graded groups of the filtration in Proposition~\ref{prop:directed-colimit-convergence} from the second page, we prove that every differential \( d_r \), for \( r\geq2 \), is zero.

\begin{proposition}
\label{prop:integral-degeneration}
For each \( r\geq2 \), the differential \( d_r\colon E^r_{p,q}\to E^r_{p-r,q+r-1} \) is zero. Consequently, \( E^2=E^\infty \).
\end{proposition}

\begin{proof}
We prove by induction on \( r\geq2 \) that \( d_r=0 \). By Equation~\eqref{eq:colimit-second-page}, \( E^2_{p,q}=0 \) for odd \( q \). If \( q \) is odd, the domain of \( d_2 \) is zero, and if \( q \) is even, its codomain \( E^2_{p-2,q+1} \) is zero. Hence \( d_2=0 \).

Fix \( r\geq3 \), and suppose that \( d_s=0 \) for every \( 2\leq s<r \). Then \( E^r=E^2 \). Equation~\eqref{eq:colimit-rational-degeneration} implies that every element of \( \operatorname{im}(d_r) \) maps to zero after tensoring with \( \mathbb Q \), so \( \operatorname{im}(d_r) \) is torsion.

If \( p>n \), then \( E^r_{p,q}=E^2_{p,q}=0 \) by Equation~\eqref{eq:colimit-horizontal-bound}. We may therefore assume that \( p\leq n \), so \( p-r<n \). If \( p-r<0 \), then \( E^r_{p-r,q+r-1}=0 \). If \( 0\leq p-r<n \), then Equation~\eqref{eq:colimit-second-page} and Theorem~\ref{thm:integer-coefficient-group-homology} identify the target \( E^r_{p-r,q+r-1}=E^2_{p-r,q+r-1} \) with either \( 0 \) or \( \mathbb Q^{(\mathfrak c)} \). Thus the target is torsion-free.

Since \( \operatorname{im}(d_r) \) is torsion and the target is torsion-free, \( \operatorname{im}(d_r)=0 \). Hence \( d_r=0 \). By induction, \( d_r=0 \) for every \( r\geq2 \). Therefore \( E^2=E^\infty \).
\end{proof}

\subsection[\texorpdfstring{Computation of the \( K \)-groups for finite-rank \( H \)}{Computation of the K-groups for finite-rank H}]{Computation of the \( K \)-groups for finite-rank \( H \)}
\label{sec:finite-rank-computation}

Proposition~\ref{prop:directed-colimit-convergence}, Proposition~\ref{prop:integral-degeneration}, Equation~\eqref{eq:colimit-second-page}, and Theorem~\ref{thm:integer-coefficient-group-homology} determine the associated graded groups of the finite filtration of \( K_*(C_u^*(G,\mathcal E)) \). We split this filtration to obtain an explicit description of the \( K \)-groups.

\begin{theorem}
\label{thm:finite-rank-k-theory}
The \( K \)-groups of \( C_u^*(G,\mathcal E) \) satisfy
\begin{equation}
\label{eq:finite-rank-k-theory}
\left(
K_0\bigl(C_u^*(G,\mathcal E)\bigr),
K_1\bigl(C_u^*(G,\mathcal E)\bigr)
\right)
\cong
\begin{cases}
\left(
\ell^\infty(G,\mathbb Z),
0
\right),
& n=0,\\[1ex]
\left(
\mathbb Q^{(\mathfrak c)},
\ell^\infty(G/H,\mathbb Z)
\right),
& n=1,\\[1ex]
\left(
\mathbb Q^{(\mathfrak c)}
\oplus
\ell^\infty(G/H,\mathbb Z),
\mathbb Q^{(\mathfrak c)}
\right),
& n\geq2 \text{ and } n \text{ is even},\\[1ex]
\left(
\mathbb Q^{(\mathfrak c)},
\mathbb Q^{(\mathfrak c)}
\oplus
\ell^\infty(G/H,\mathbb Z)
\right),
& n\geq3 \text{ and } n \text{ is odd}.
\end{cases}
\end{equation}
The isomorphism in Equation~\eqref{eq:finite-rank-k-theory} is canonical when \( n=0 \) and noncanonical when \( n\geq1 \). Before choosing the ordered basis of \( H \), the associated graded group \( F_nK_{n\bmod 2}/F_{n-1}K_{n\bmod 2} \) is canonically isomorphic to \( \ell^\infty(G/H,\mathbb Z)\otimes_{\mathbb Z}\Lambda^nH \).
\end{theorem}

\begin{proof}
Fix \( N\in\mathbb Z/2\mathbb Z \). Proposition~\ref{prop:directed-colimit-convergence} gives the finite filtration \( 0=F_{-1}K_N\subseteq F_0K_N\subseteq\cdots\subseteq F_nK_N=K_N(C_u^*(G,\mathcal E)) \). By Proposition~\ref{prop:integral-degeneration}, Equation~\eqref{eq:colimit-second-page}, and Theorem~\ref{thm:integer-coefficient-group-homology}, the successive quotients satisfy
\[
F_pK_N/F_{p-1}K_N
\cong
\begin{cases}
\mathbb Q^{(\mathfrak c)},
& p<n \text{ and } p\equiv N\pmod 2,\\
\ell^\infty(G/H,\mathbb Z),
& p=n \text{ and } n\equiv N\pmod 2,\\
0,
& \text{otherwise}.
\end{cases}
\]

We prove by induction that, for every \( -1\leq p<n \),
\[
F_pK_N
\cong
\bigoplus_{\substack{0\leq j\leq p\\j\equiv N\pmod 2}}
\mathbb Q^{(\mathfrak c)}.
\]
For \( p=-1 \), both sides are zero. Fix \( 0\leq p<n \), and suppose that the assertion holds for \( F_{p-1}K_N \). From the filtration, we obtain the short exact sequence \( 0\to F_{p-1}K_N\to F_pK_N\to F_pK_N/F_{p-1}K_N\to0 \). By the inductive hypothesis, \( F_{p-1}K_N \) is isomorphic to a finite direct sum of divisible groups. Hence \( F_{p-1}K_N \) is divisible and therefore injective. The inclusion \( F_{p-1}K_N\to F_pK_N \) admits a retraction, so the sequence splits. If \( p\not\equiv N\pmod 2 \), then the quotient is zero and \( F_pK_N=F_{p-1}K_N \). If \( p\equiv N\pmod 2 \), then \( F_pK_N\cong F_{p-1}K_N\oplus\mathbb Q^{(\mathfrak c)} \). This completes the induction.

Therefore
\[
F_{n-1}K_N
\cong
\bigoplus_{\substack{0\leq p<n\\p\equiv N\pmod 2}}
\mathbb Q^{(\mathfrak c)}.
\]
Every nonzero finite direct sum of copies of \( \mathbb Q^{(\mathfrak c)} \) is isomorphic to \( \mathbb Q^{(\mathfrak c)} \). An even integer \( p \) with \( 0\leq p<n \) exists exactly when \( n\geq1 \), and an odd integer \( p \) with \( 0\leq p<n \) exists exactly when \( n\geq2 \). Hence
\[
F_{n-1}K_0
\cong
\begin{cases}
0,
& n=0,\\
\mathbb Q^{(\mathfrak c)},
& n\geq1,
\end{cases}
\qquad
F_{n-1}K_1
\cong
\begin{cases}
0,
& n\leq1,\\
\mathbb Q^{(\mathfrak c)},
& n\geq2.
\end{cases}
\]

If \( n\not\equiv N\pmod 2 \), then \( F_nK_N/F_{n-1}K_N=0 \), so \( K_N(C_u^*(G,\mathcal E))=F_{n-1}K_N \). If \( n\equiv N\pmod 2 \), we obtain the short exact sequence \( 0\to F_{n-1}K_N\to K_N(C_u^*(G,\mathcal E))\to\ell^\infty(G/H,\mathbb Z)\to0 \). When \( F_{n-1}K_N=0 \), the map \( K_N(C_u^*(G,\mathcal E))\to\ell^\infty(G/H,\mathbb Z) \) is an isomorphism. When \( F_{n-1}K_N\cong\mathbb Q^{(\mathfrak c)} \), the group \( F_{n-1}K_N \) is divisible and therefore injective, so the sequence splits.

For \( n=0 \), we have \( G/H=G \), and the filtration gives \( K_0(C_u^*(G,\mathcal E))\cong\ell^\infty(G,\mathbb Z) \) and \( K_1(C_u^*(G,\mathcal E))=0 \). For \( n=1 \), the preceding computation gives \( K_0(C_u^*(G,\mathcal E))\cong\mathbb Q^{(\mathfrak c)} \) and \( K_1(C_u^*(G,\mathcal E))\cong\ell^\infty(G/H,\mathbb Z) \). For \( n\geq2 \), the integers \( p \) with \( 0\leq p<n \) include both parities, so the preceding computation gives a \( \mathbb Q^{(\mathfrak c)} \) summand in both \( K_0(C_u^*(G,\mathcal E)) \) and \( K_1(C_u^*(G,\mathcal E)) \). If \( n \) is even, the quotient in filtration degree \( n \) gives the \( \ell^\infty(G/H,\mathbb Z) \) summand in \( K_0(C_u^*(G,\mathcal E)) \). If \( n \) is odd, it gives the \( \ell^\infty(G/H,\mathbb Z) \) summand in \( K_1(C_u^*(G,\mathcal E)) \). This gives Equation~\eqref{eq:finite-rank-k-theory}.

For \( p<n \), the resulting decompositions depend on the isomorphisms in Theorem~\ref{thm:integer-coefficient-group-homology} and on the chosen splittings of the filtration. The identification \( F_nK_{n\bmod 2}/F_{n-1}K_{n\bmod 2}\cong\ell^\infty(G/H,\mathbb Z) \) depends on the chosen ordered basis of \( H \).
\end{proof}

\subsection{Low-rank cases}

We record Theorem~\ref{thm:finite-rank-k-theory} in ranks \( 0 \), \( 1 \), and \( 2 \). For \( G=H=\mathbb Z^n \), our formulas in ranks \( 1 \) and \( 2 \) agree with the computations of Kato, Kishimoto, and Tsutaya and determine the dimensions of the rational vector spaces in those computations \cite{Kato_2021}.

When \( n=0 \), we have \( H=\{0\} \), and \( \mathcal E \) consists precisely of the subsets of the diagonal. Therefore \( C_u^*(G,\mathcal E)=\ell^\infty(G) \), so \( K_0(C_u^*(G,\mathcal E))\cong\ell^\infty(G,\mathbb Z) \) and \( K_1(C_u^*(G,\mathcal E))=0 \).

Suppose that \( n=1 \), and choose a generator \( h \) of \( H \). The Koszul complex is
\[
0
\longrightarrow
\ell^\infty(G,\mathbb Z)
\xrightarrow{\,f\mapsto f-f\cdot h\,}
\ell^\infty(G,\mathbb Z)
\longrightarrow
0.
\]
Its homology groups satisfy
\[
H_1\bigl(H;\ell^\infty(G,\mathbb Z)\bigr)
\cong
\ell^\infty(G/H,\mathbb Z),
\qquad
H_0\bigl(H;\ell^\infty(G,\mathbb Z)\bigr)
=
\frac{\ell^\infty(G,\mathbb Z)}
{\{f-f\cdot h:f\in\ell^\infty(G,\mathbb Z)\}}.
\]
Theorem~\ref{thm:integer-coefficient-group-homology} identifies \( H_0(H;\ell^\infty(G,\mathbb Z)) \) noncanonically with \( \mathbb Q^{(\mathfrak c)} \). Theorem~\ref{thm:finite-rank-k-theory} therefore gives
\[
K_0\bigl(C_u^*(G,\mathcal E)\bigr)
\cong
\mathbb Q^{(\mathfrak c)},
\qquad
K_1\bigl(C_u^*(G,\mathcal E)\bigr)
\cong
\ell^\infty(G/H,\mathbb Z).
\]
For \( G=H=\mathbb Z \), the quotient \( G/H \) is trivial. Hence \( K_0(C_u^*(\mathbb Z,\mathcal E))\cong\mathbb Q^{(\mathfrak c)} \) and \( K_1(C_u^*(\mathbb Z,\mathcal E))\cong\mathbb Z \). This agrees with \cite[Proposition~6.1]{Kato_2021} and strengthens \cite[Proposition~6.3]{Kato_2021} by identifying the dimension of the rational vector space as \( \mathfrak c \).

Suppose that \( n=2 \), and choose an ordered basis \( h_1,h_2 \) of \( H \). The Koszul complex is
\[
0
\longrightarrow
\ell^\infty(G,\mathbb Z)
\xrightarrow{\partial_2}
\ell^\infty(G,\mathbb Z)^2
\xrightarrow{\partial_1}
\ell^\infty(G,\mathbb Z)
\longrightarrow
0,
\]
where \( \partial_1(a,b)=a-a\cdot h_1+b-b\cdot h_2 \) and \( \partial_2(c)=(-(c-c\cdot h_2),c-c\cdot h_1) \). Theorem~\ref{thm:integer-coefficient-group-homology} gives noncanonical isomorphisms
\[
H_0\bigl(H;\ell^\infty(G,\mathbb Z)\bigr)
\cong
H_1\bigl(H;\ell^\infty(G,\mathbb Z)\bigr)
\cong
\mathbb Q^{(\mathfrak c)}.
\]
Lemma~\ref{lem:top-koszul-homology} gives the canonical form
\[
H_2\bigl(H;\ell^\infty(G,\mathbb Z)\bigr)
\cong
\ell^\infty(G/H,\mathbb Z)\otimes_{\mathbb Z}\Lambda^2H.
\]
The chosen ordered basis identifies this group with \( \ell^\infty(G/H,\mathbb Z) \). Theorem~\ref{thm:finite-rank-k-theory} therefore gives
\[
K_0\bigl(C_u^*(G,\mathcal E)\bigr)
\cong
\mathbb Q^{(\mathfrak c)}
\oplus
\ell^\infty(G/H,\mathbb Z),
\qquad
K_1\bigl(C_u^*(G,\mathcal E)\bigr)
\cong
\mathbb Q^{(\mathfrak c)}.
\]
For \( G=H=\mathbb Z^2 \), the quotient \( G/H \) is trivial. Hence \( K_0(C_u^*(\mathbb Z^2,\mathcal E))\cong\mathbb Q^{(\mathfrak c)}\oplus\mathbb Z \) and \( K_1(C_u^*(\mathbb Z^2,\mathcal E))\cong\mathbb Q^{(\mathfrak c)} \). This agrees with \cite[Lemmas~7.2, 7.6, and 7.7]{Kato_2021} and strengthens Lemma~7.7 by identifying both rational dimensions as \( \mathfrak c \).

\section{Applications}

\subsection{\texorpdfstring{Uniform \( L^p \)-Roe algebras}{Uniform Lp-Roe algebras}}

For a uniformly locally finite extended metric space \( (X,d) \) and \( p\in(1,\infty) \), the uniform \( L^p \)-Roe algebra \( B_u^p(X,d) \) is the operator norm closure in \( \mathcal B(\ell^p(X)) \) of the algebra of all bounded operators of finite propagation \cite[p.~634]{austad2025polynomial}. For \( p=2 \), \( B_u^2(X,d) \) is the uniform Roe algebra of the bounded coarse structure induced by \( d \).

A central problem in the \( K \)-theory of \( L^p \)-operator algebras is to determine when the \( K \)-groups are independent of the exponent \( p \), and this problem has also been studied for \( L^p \)-Roe algebras \cite[Question~3.1]{austad2025polynomial}. For a uniformly locally finite extended metric space whose ball cardinalities admit a uniform polynomial bound, the \( K \)-groups of \( B_u^p(X,d) \) are independent of \( p\in(1,\infty) \) \cite[Corollary~5.9]{austad2025polynomial}.

In this subsection, we construct an extended metric \( d_H \) on \( G \) that induces \( \mathcal E \), makes \( (G,d_H) \) uniformly locally finite, and gives a uniform polynomial bound on its ball cardinalities. Since \( B_u^2(G,d_H)=C_u^*(G,\mathcal E) \), Corollary~5.9 of \cite{austad2025polynomial} transfers the computation of Theorem~\ref{thm:finite-rank-k-theory} from \( p=2 \) to every \( p\in(1,\infty) \).

Let \( |\cdot|_H \) denote the word length on \( H \) associated with the basis \( e_1,\ldots,e_n \). Define \( d_H\colon G\times G\to[0,\infty] \) by
\[
d_H(g,k)=
\begin{cases}
|g-k|_H, & g-k\in H,\\
\infty, & g-k\notin H.
\end{cases}
\]
On each \( H \)-coset, the restriction of \( d_H \) is the translated word metric on \( H \), while points in distinct \( H \)-cosets have infinite distance. Thus \( d_H \) is an extended metric on \( G \).

\begin{lemma}
\label{lem:lp-extended-metric}
The extended metric space \( (G,d_H) \) is uniformly locally finite, has polynomial growth, and the coarse structure induced by \( d_H \) is \( \mathcal E \).
\end{lemma}

\begin{proof}
For \( g\in G \) and \( R\geq0 \), we have 
\( B_{d_H}(g,R)=g+\{h\in H:|h|_H\leq R\} \). Hence 
\[ 
|B_{d_H}(g,R)|
\leq
(2\lfloor R\rfloor+1)^n
\leq
(2R+1)^n 
\] 
for every \( g\in G \) and \( R\geq0 \). Thus \( (G,d_H) \) is uniformly locally finite and has polynomial growth.

For \( R\geq0 \), the pairs \( (g,k) \) satisfying \( d_H(g,k)\leq R \) are exactly those for which \( g-k\in H \) and \( |g-k|_H\leq R \). Since \( \{h\in H:|h|_H\leq R\} \) is finite, the set \( \{(g,k)\in G\times G:d_H(g,k)\leq R\} \) is one of the sets used to generate \( \mathcal E \). Every element of the coarse structure induced by \( d_H \) is contained in \( \{(g,k)\in G\times G:d_H(g,k)\leq R\} \) for some \( R\geq0 \), so it belongs to \( \mathcal E \). Conversely, for every finite \( F\subseteq H \), some \( R\geq0 \) satisfies \( |h|_H\leq R \) for all \( h\in F \). Hence each generator \( \{(g,k)\in G\times G:g-k\in F\} \) of \( \mathcal E \) is contained in \( \{(g,k)\in G\times G:d_H(g,k)\leq R\} \). Therefore the coarse structure induced by \( d_H \) is \( \mathcal E \).
\end{proof}

For \( p\in(1,\infty) \), let \( B_u^p(G,d_H) \) denote the uniform \( L^p \)-Roe algebra of \( (G,d_H) \). By Lemma~\ref{lem:lp-extended-metric} and \cite[Corollary~5.9]{austad2025polynomial}, the \( K \)-groups of \( B_u^p(G,d_H) \) are independent of \( p\in(1,\infty) \). Since \( B_u^2(G,d_H)=C_u^*(G,\mathcal E) \), Theorem~\ref{thm:finite-rank-k-theory} gives the following computation.

\begin{corollary}
\label{cor:lp-uniform-roe}
For every \( p\in(1,\infty) \), the \( K \)-groups of \( B_u^p(G,d_H) \) satisfy
\[
\left(
K_0\bigl(B_u^p(G,d_H)\bigr),
K_1\bigl(B_u^p(G,d_H)\bigr)
\right)
\cong
\begin{cases}
\left(
\ell^\infty(G,\mathbb Z),
0
\right),
& n=0,\\[1ex]
\left(
\mathbb Q^{(\mathfrak c)},
\ell^\infty(G/H,\mathbb Z)
\right),
& n=1,\\[1ex]
\left(
\mathbb Q^{(\mathfrak c)}
\oplus
\ell^\infty(G/H,\mathbb Z),
\mathbb Q^{(\mathfrak c)}
\right),
& n\geq2 \text{ even},\\[1ex]
\left(
\mathbb Q^{(\mathfrak c)},
\mathbb Q^{(\mathfrak c)}
\oplus
\ell^\infty(G/H,\mathbb Z)
\right),
& n\geq3 \text{ odd}.
\end{cases}
\]
\end{corollary}

\begin{proof}
By Lemma~\ref{lem:lp-extended-metric} and \cite[Corollary~5.9]{austad2025polynomial}, the groups \( K_*(B_u^p(G,d_H)) \) are independent of \( p\in(1,\infty) \). In particular, \( K_*(B_u^p(G,d_H))\cong K_*(B_u^2(G,d_H)) \). Since \( d_H \) induces \( \mathcal E \), we have \( B_u^2(G,d_H)=C_u^*(G,\mathcal E) \). The result follows from Theorem~\ref{thm:finite-rank-k-theory}.
\end{proof}

\subsection{Virtual persistence diagrams}

For a metric pair \( (X,d,A) \) with \( A\neq\varnothing \), let \( D(X) \) denote the free commutative monoid of finite formal sums of points of \( X \), and let \( D(X,A)=D(X)/D(A) \) be the commutative monoid of persistence diagrams over \( (X,d,A) \) \cite[Sections~2.2--2.3]{bubenik2022virtual}. The group of virtual persistence diagrams is the Grothendieck group completion \( K(X,A)=K(D(X,A)) \), whose elements have the form \( \alpha-\beta \) with \( \alpha,\beta\in D(X,A) \) \cite[Section~2.5 and Definition~4.10]{bubenik2022virtual}. The \( 1 \)-Wasserstein metric extends to the translation-invariant metric
\[
\rho(\alpha-\beta,\gamma-\delta)
=
W_1[d](\alpha+\delta,\gamma+\beta),
\qquad
\alpha,\beta,\gamma,\delta\in D(X,A),
\]
on \( K(X,A) \) \cite[Corollary~4.9 and Definition~4.10]{bubenik2022virtual}. When \( X \) is finite, \( K(X,A) \) is the free abelian group on \( X\setminus A \), so if \( m=|X\setminus A| \), then \( K(X,A)\cong\mathbb Z^m \) \cite[Definition~4.10 and the subsequent paragraph]{bubenik2022virtual}.

Recent work has begun to develop harmonic and functional-analytic methods directly on virtual persistence diagram groups \cite{fanningaktas2025rkhsvpd, fanningaktas2026banachrkhs, fanningaktas2026randomwalks}. These works develop analytic structures from the abelian group structure and translation-invariant metric geometry of virtual persistence diagram groups \cite{fanningaktas2025rkhsvpd, fanningaktas2026banachrkhs, fanningaktas2026randomwalks}.

For a finite metric pair \( (X,d,A) \), let \( \mathcal E_\rho \) denote the coarse structure induced by \( \rho \) on \( K(X,A) \). In this subsection, we show that \( \mathcal E_\rho \) is generated by the entourages determined by finite subsets of \( K(X,A) \). Consequently, Theorem~\ref{thm:finite-rank-k-theory} computes the \( K \)-groups of \( C_u^*(K(X,A),\mathcal E_\rho) \), and their isomorphism classes depend only on \( m=|X\setminus A| \).

\begin{lemma}
\label{lem:virtual-persistence-coarse-structure}
The sets
\[
\{(\alpha,\beta)\in K(X,A)\times K(X,A):\alpha-\beta\in F\},
\]
where \( F \) ranges over the finite subsets of \( K(X,A) \), generate \( \mathcal E_\rho \). In particular, \( (K(X,A),\mathcal E_\rho) \) is uniformly locally finite.
\end{lemma}

\begin{proof}
By \cite[Corollary~7.10]{bubenik2022virtual}, the canonical map \( (K(X,A),\rho)\to V(X,A) \) is an isometric embedding when \( V(X,A) \) carries its \( 1 \)-Wasserstein norm. By \cite[pp.~468, 472]{bubenik2022virtual}, \( V(X,A)=V(X)/V(A) \), where \( V(X) \) is the free real vector space on \( X \). Hence the canonical images of the elements of \( X\setminus A \) form a real basis of \( V(X,A) \), so \( \dim_{\mathbb R}V(X,A)=m \).

Equip \( V(X,A) \) with the coordinate \( \ell^1 \)-norm for this basis, and let \( |\cdot| \) denote its restriction to \( K(X,A) \). This restriction equals the word length for the canonical images of the elements of \( X\setminus A \) and their inverses. Since \( V(X,A) \) is finite-dimensional, the coordinate \( \ell^1 \)-norm and the \( 1 \)-Wasserstein norm are equivalent. We may therefore choose constants \( c,C>0 \) such that \( c|\alpha|\leq \rho(\alpha,0)\leq C|\alpha| \) for every \( \alpha\in K(X,A) \). Since \( \rho \) is translation-invariant, we obtain \( c|\alpha-\beta|\leq \rho(\alpha,\beta)\leq C|\alpha-\beta| \) for all \( \alpha,\beta\in K(X,A) \). Hence \( \rho \) and the word metric induce the same coarse structure on \( K(X,A) \).

For \( R\geq0 \), the set \( \{\gamma\in K(X,A):|\gamma|\leq R\} \) is finite, and the pairs at word-metric distance at most \( R \) are exactly the pairs \( (\alpha,\beta) \) whose difference belongs to this set. Conversely, every finite subset \( F\subseteq K(X,A) \) lies in a word-metric ball of finite radius. Thus the sets in the lemma statement generate \( \mathcal E_\rho \), and every entourage in \( \mathcal E_\rho \) is contained in one of these sets. For each finite \( F\subseteq K(X,A) \) and each \( \alpha\in K(X,A) \), the fiber at \( \alpha \) is \( \alpha-F \) and therefore has cardinality \( |F| \). Hence \( (K(X,A),\mathcal E_\rho) \) is uniformly locally finite.
\end{proof}

Lemma~\ref{lem:virtual-persistence-coarse-structure} and Theorem~\ref{thm:finite-rank-k-theory} give the following computation.

\begin{corollary}
\label{cor:virtual-persistence-k-theory}
Let \( (X,d) \) be a finite metric space, let \( \varnothing\neq A\subseteq X \), and set \( m=|X\setminus A| \). Then
\[
\left(
K_0\bigl(C_u^*(K(X,A),\mathcal E_\rho)\bigr),
K_1\bigl(C_u^*(K(X,A),\mathcal E_\rho)\bigr)
\right)
\cong
\begin{cases}
\left(
\mathbb Z,
0
\right),
& m=0,\\[1ex]
\left(
\mathbb Q^{(\mathfrak c)},
\mathbb Z
\right),
& m=1,\\[1ex]
\left(
\mathbb Q^{(\mathfrak c)}
\oplus
\mathbb Z,
\mathbb Q^{(\mathfrak c)}
\right),
& m\geq2 \text{ even},\\[1ex]
\left(
\mathbb Q^{(\mathfrak c)},
\mathbb Q^{(\mathfrak c)}
\oplus
\mathbb Z
\right),
& m\geq3 \text{ odd}.
\end{cases}
\]
\end{corollary}

\begin{proof}
By Lemma~\ref{lem:virtual-persistence-coarse-structure} and \( K(X,A)\cong\mathbb Z^m \), Theorem~\ref{thm:finite-rank-k-theory} applies with \( G=H=K(X,A) \). Since \( G/H \) is trivial, \( \ell^\infty(G/H,\mathbb Z)\cong\mathbb Z \). The stated formulas follow.
\end{proof}

Thus, for a finite metric space \( (X,d) \) and a nonempty subset \( A\subseteq X \), the isomorphism classes of the \( K \)-groups of \( C_u^*(K(X,A),\mathcal E_\rho) \) depend only on \( |X\setminus A| \).

\section*{Declarations}

\begin{itemize}
% \item \textbf{Competing Interests} The authors declare that they have no competing interests.
% \item \textbf{Funding} This research received no external funding.
% \item \textbf{Data Availability} Not applicable.
% \item \textbf{Code Availability} The implementation used in this work is available at \url{https://github.com/cfanning8/Virtual_Persistence_RKHS}.
\item \textbf{Authors' Contributions}
C.F. developed the theoretical framework, proved the main results, and wrote the manuscript. M.E.A. advised the project and provided feedback on the framework and manuscript.
\item \textbf{AI Usage} The authors used an AI language model to rephrase dense analytic descriptions and to proofread and correct the language and grammar during manuscript preparation. The authors assume responsibility for all content in this publication.
\end{itemize}

\bibliography{sn-bibliography}% common bib file

\end{document}